\documentclass[12pt]{amsart}
\usepackage[all]{xy}
\usepackage[mathscr]{euscript}
\usepackage{amsmath,amssymb,amsthm}
\usepackage{hyperref}
\hypersetup{colorlinks=true,urlcolor=blue,citecolor=blue,linkcolor=blue}
\usepackage{courier}
\usepackage{amscd,latexsym,longtable,diagram,picture}

\usepackage{ragged2e}
\usepackage[normalem]{ulem}
\usepackage{graphicx, tabularx}
\usepackage{array}
\usepackage{color}
\usepackage{enumerate}
\usepackage{nicefrac}
\usepackage{listings}
\usepackage{bm,bbm,mathrsfs}

\usepackage{epsfig,xcolor}
\usepackage[shortlabels]{enumitem}
\usepackage{indentfirst, setspace}
\usepackage{tikz-cd}

\usepackage[margin=1.0in]{geometry}
\usepackage{booktabs}
\usepackage[T1]{fontenc}
\usepackage[utf8]{inputenc}
\usepackage{lmodern}
\usepackage{microtype}
\microtypesetup{nopatch=footnote}

\usepackage{silence}
\allowdisplaybreaks
\usepackage{caption}
\usepackage{cancel}

\usetikzlibrary{calc,matrix,arrows,decorations.markings}

\newcommand{\Z}{\mathbb{Z}}
\newcommand{\Q}{\mathbb{Q}}

\newcommand{\mat}{\begin{pmatrix}}
\newcommand{\emat}{\end{pmatrix}}

\providecommand{\qfdata}[2]{}
\renewcommand{\qfdata}[2]{\ensuremath{\begin{pmatrix}#1\\#2\end{pmatrix}}}
\providecommand{\fkdata}[2]{}
\renewcommand{\fkdata}[2]{\ensuremath{\begin{pmatrix}#1\\#2\end{pmatrix}}}
\providecommand{\vct}[1]{}
\renewcommand{\vct}[1]{\ensuremath{(#1)}}

\newtheorem{theorem}{Theorem}[section]
\newtheorem{proposition}[theorem]{Proposition}

\newtheorem{lemma}[theorem]{Lemma}
\newtheorem{rmk}[theorem]{Remark}

\newtheorem*{theorem*}{Theorem}

\theoremstyle{remark}
\newtheorem*{rmk*}{\textit{Remark}}

\author{Bo-Hae Im and Minseo Shin}

\address{
Dept. of Mathematical Sciences, KAIST,
291 Daehak-ro, Yuseong-gu,
Daejeon 34141, South Korea
}
\email{bhim@kaist.ac.kr}

\address{
Dept. of Mathematical Sciences, KAIST,
291 Daehak-ro, Yuseong-gu,
Daejeon 34141, South Korea
}
\email{minseo74da@kaist.ac.kr}

\date{\today}
\thanks{Bo-Hae Im was supported by Basic Science Research Program through the National Research Foundation of Korea(NRF) grant funded by the Korea government(MSIT)(NRF-2023R1A2C1002385, or RS-2023-NR076333).}

\begin{document}

\setlist[enumerate]{label=\textup{(\alph*)},ref=\alph*}
\theoremstyle{plain}
\newtheorem{maintheorem}{Theorem}
\renewcommand{\themaintheorem}{\Alph{maintheorem}}
\theoremstyle{remark}

\title[]
{Tunnell-type criteria for variants of the congruent number problem}

\begin{abstract}
We study the \(\theta\)-congruent number problem for \(\cos\theta=\pm3/5\) and \(\pm4/5\) using the generalized theta series construction of Sirolli--Tornar\'ia \cite{ST}. We describe its specialization to newforms of weight \(2\) over \(\mathbb Q\) with nontrivial square-free odd part of the level, and explain the reduction of quadratic twists to odd fundamental discriminants.
The same construction gives an effective procedure for every \(\theta\)-congruent number problem with nonzero rational cosine.
For the four angles, we construct explicit forms of weight \(3/2\) whose Fourier coefficients determine the central \(L\)-values of the associated elliptic curves. This gives Tunnell-type criteria for every positive square-free integer: a nonzero coefficient implies non-\(\theta\)-congruence unconditionally, and the converse holds assuming the Birch--Swinnerton-Dyer conjecture.
We also prove unconditional non-\(\theta\)-congruence for primes in explicit arithmetic progressions.
\end{abstract}

\maketitle

\section{Introduction}
A positive integer \(n\) is called a \emph{congruent number} if it occurs as the area of a right triangle with rational side lengths. The classical congruent number problem asks whether a given positive integer is a congruent number or not. This question is equivalent to determining whether the associated quadratic twist \(E^{(n)}: y^2=x^3-n^2x\) of the elliptic curve
\begin{align}\label{eq:orig-curve}
E:\quad y^2=x^3-x
\end{align}
has positive Mordell--Weil rank over \(\mathbb Q\). The curves \(E^{(n)}\) are commonly called congruent number elliptic curves; \cite{TY} provides a survey on this problem and its variants.

There are several analytic and algebraic approaches to this problem.
The main analytic approaches are naturally divided according to the sign of the functional equation of the Hasse--Weil $L$-function \(L(E^{(n)},s)\). When the sign is \(+1\), Tunnell's theorem \cite{Tu} gives a striking criterion for this problem in terms of the Fourier coefficients of certain modular forms of weight \(3/2\). 
More precisely, Tunnell constructs $q$-series whose coefficients are linear combinations of representation numbers of ternary quadratic forms, and these coefficients determine the central \(L\)-values of the quadratic twists $E^{(n)}$ of \eqref{eq:orig-curve}.
Assuming the Birch--Swinnerton-Dyer conjecture (BSD), this provides an effective criterion for deciding whether a given positive integer is congruent or not. By \cite[Theorem~3]{Tu}, we get the following:

\begin{theorem}[Tunnell]\label{thm:Tunnell}
Let $n$ be a positive odd square-free integer. Define \(r_n(a, b, c) = |\{(x, y, z) \in \Z^3 : ax^2 + by^2 + cz^2 = n\}|.\)
\begin{enumerate}[\normalfont (a)]
\item If $2r_n(2, 1, 32) - r_n(2, 1, 8) \neq 0$, then $n$ is not congruent.
\item  If $2r_n(4, 1, 32) - r_n(4, 1, 8) \neq 0$, then $2n$ is not congruent.
\end{enumerate}
If the BSD conjecture is true, then the converses of both (a) and (b) are true.
\end{theorem}

When the sign of the functional equation is \(-1\), it forces \(L(E^{(n)},1)=0\), so BSD predicts positive Mordell--Weil rank for $E^{(n)}$. 
For unconditional results, one instead seeks to produce rational points of infinite order.
In the classical congruent number problem, explicit constructions of such points go back to Heegner \cite{He}. The Gross--Zagier formula \cite{GZ} later related the N\'eron--Tate heights of Heegner points to central derivatives of \(L\)-functions, thereby connecting the non-torsion of these points with analytic rank one. Combined with Kolyvagin's Euler-system method \cite{Ko}, a non-torsion Heegner point yields algebraic rank one and the finiteness of the Tate--Shafarevich group in the corresponding analytic-rank-one setting. Such point constructions for congruent number elliptic curves were further developed by Monsky \cite{Mo}, and Tian \cite{Ti} extended these methods to families with arbitrarily many prime factors. Related criteria using genus periods and genus points were obtained by Tian--Yuan--Zhang \cite{TYZ}.

Let \(0<\theta<\pi\) be a non-right angle with rational cosine, say \(\cos\theta=s/r\), where \(r,s\in\mathbb Z\), \(r>0\), \((r,s)=1\), and \(0<|s|<r\). A positive integer \(n\) is called a \(\theta\)-congruent number if there is a triangle with rational sides and an angle \(\theta\) whose area is \(n\sqrt{r^2-s^2}\). Fujiwara \cite{Fu} associated to this problem the elliptic curve
\begin{equation}\label{eq:thetacong}
    E_{r,s}:\quad
    y^2=x^3+2sx^2-(r^2-s^2)x,
\end{equation}
now called the \(\theta\)-congruent elliptic curve, and proved that $n$ is $\theta$-congruent if and only if the quadratic twist $E_{r, s}^{(n)}$ by $n$ of $E_{r, s}$ given in \eqref{eq:thetacong} has a rational point outside its rational $2$-torsion subgroup.
Thus, as in the classical case, the problem becomes one about quadratic twists of a fixed elliptic curve.

For general rational \(\cos\theta\), Kan \cite{Kan} parametrized
\(\theta\)-congruent squareclasses and obtained several results for prime
values of \(n\). 
This parametric direction was developed further by Im--Kim \cite{IK}, who constructed families whose square-free parts contain arbitrarily many, and even prescribed, prime factors and determined the probability that a given prime occurs as such a factor.

The cases \(\theta=\pi/3\) and \(2\pi/3\) have received particular attention. Most directly related to the present work are Yoshida's analogues \cite{Y1,Y2} of Tunnell's theorem. Yoshida constructed modular forms of weight \(3/2\) whose coefficients are expressed in terms of ternary quadratic forms and whose squares are related to central \(L\)-values by Waldspurger's formula. A related Tunnell-type treatment for angles with irrational cosine was given by Dimabayao and Purkait \cite{DP}, who applied the Shimura correspondence and Waldspurger's theorem to the cases \(\cos\theta=\pm\sqrt2/2\). Heegner-point methods have also been applied to the special angles \(\pi/3\) and \(2\pi/3\) \cite{HK}, while \(2\)-descent has produced families of non-\(\theta\)-congruent numbers \cite{GLN}.

This paper studies variants of the classical congruent number problem following the general strategy of Tunnell's work.
The conceptual mechanisms behind Tunnell's theorem are the Shimura lift \cite{Sh} and Waldspurger's theorem \cite{Wa}. The Shimura lift is a Hecke-equivariant map from modular forms of half-integral weight to modular forms of integral weight. For a cuspidal form \(g\) of weight \(3/2\) corresponding to a cuspidal newform of weight \(2\), Waldspurger's theorem relates the squares of its Fourier coefficients to central \(L\)-values of quadratic twists, subject to local conditions. The Hecke correspondence is taken at primes away from the level.
Tunnell's contribution was to make this relation explicit for the congruent number elliptic curve~\(E\) of \eqref{eq:orig-curve}:
he identified the forms of weight \(3/2\) whose Shimura lifts are the newform of weight \(2\) corresponding to~\(E\), and expressed their coefficients in terms of representation numbers of concrete ternary quadratic forms, as in Theorem~\ref{thm:Tunnell}.

Thus, constructing explicit half-integral-weight forms is essential for obtaining Tunnell-type results for both the classical and the $\theta$-congruent number problems. In this paper, in order to construct the required forms of weight $3/2$ systematically for a fixed angle $\theta$ with $\cos\theta\in\mathbb{Q}$, we use Brandt modules \cite[pp.~116--139]{Gr} and generalized theta series.
This approach originates in Gross's construction \cite{Gr} of half-integral-weight forms associated, under the Shimura correspondence, with newforms of weight $2$.
For a newform of prime level $N$, Gross used theta series attached to trace-zero ternary lattices arising from maximal orders in the definite quaternion algebra ramified at $N$ and $\infty$. When $L(f,1)\neq0$, his construction gives a nonzero form of weight $3/2$ whose Fourier coefficients determine the central
$L$-values of negative fundamental discriminant twists satisfying a prescribed quadratic residue symbol condition at $N$.
Gross's construction was subsequently generalized by B\"ocherer--Schulze-Pillot \cite{BSP1,BSP2} to odd square-free levels and certain higher weights, although there were still restrictions on the twists detected by each theta series. At odd prime level, Mao--Rodr\'iguez-Villegas--Tornar\'ia \cite{MRVT} introduced local weight functions and computed generalized theta series for both signs of the twisting discriminant, without requiring \(L(f,1)\neq0\).
Pacetti--Tornar\'ia \cite{PT} developed corresponding examples at composite levels.
More recently, Sirolli--Tornar\'ia \cite{ST} gave a systematic construction of half-integral-weight forms over totally real fields.  For a suitable Hilbert newform \(f\), they first divide its quadratic twists into families according to the signs of the corresponding quadratic characters at the primes dividing the level and at the real places. For each admissible type of root number \(+1\), they construct a generalized theta series whose coefficients at fundamental discriminants determine the corresponding central \(L\)-values. Its other coefficients need not vanish.
We apply this construction over \(\mathbb Q\) to the newforms of weight \(2\) associated with the elliptic curves \(E_{r,s}\) given in \eqref{eq:thetacong}.

The contribution of this paper is twofold. First, we formulate the effective Waldspurger-type construction in a form directly applicable to \(\theta\)-congruent number problems. Second, for the four angles
\begin{equation}\label{eq:cos-values}
\cos\theta=\pm\frac35,\qquad
\cos\theta=\pm\frac45,
\end{equation}
we determine the relevant forms of weight \(3/2\), their levels, characters, and coefficient formulas, and derive explicit Tunnell-type criteria for the corresponding \(\theta\)-congruent number problems.
Moreover, we obtain infinite families of non-\(\theta\)-congruent primes for angles in \eqref{eq:cos-values}.

The two principal conclusions are as follows.
For \(j\in\{3,4\}\) and \(\mu\in\{1,-1,2,-2\}\), let
\(f_j^{(\mu)}\) be the normalized newform attached to
\(E_{5,j}^{(\mu)}\), of level \(N_j^{(\mu)}\).
Fix \(\cos\theta=\sigma j/5\), with \(\sigma\in\{1,-1\}\), and
write a positive square-free integer as \(n=2^e m\), with
\(e\in\{0,1\}\) and \(m\) odd.
Put \(\eta_m=(-1)^{(m-1)/2}\), \(D_0=\eta_m m\), and
\(\mu=\sigma2^e\eta_m\).
We include \(D_0=1\) among the fundamental discriminants, with trivial
quadratic character.

We now state our two main theorems precisely:

\begin{maintheorem}[Tunnell-type criterion]
\label{thm:main-A}
If the twist \(f_j^{(\mu)}\otimes\chi_{D_0}\) has root number
\(-1\), then \(L(E_{5,\sigma j}^{(n)},1)=0\).
For root number \(+1\), the rule in Section~\ref{sec:tunnell-criteria}
selects a unique row \(r\) of Table~\ref{tab:metadata-three} for
\(j=3\), or Table~\ref{tab:metadata-four} for \(j=4\).
Write \(G_r(q)=\sum_{a\geq1}a_r(a)q^a\).
There is an explicit constant \(\kappa_r>0\), depending only on the
row and its normalization, such that
\[
  L(E_{5,\sigma j}^{(n)},1)
  =2^{\omega(D_0,N_j^{(\mu)})}\kappa_r
    \frac{|a_r(m)|^2}{\sqrt m},
\]
where \(\omega(D_0,N)\) counts the primes dividing both \(D_0\) and
\(N\).
In particular, a nonzero \(a_r(m)\) implies that \(n\) is not
\(\theta\)-congruent.
Assuming the rank part of BSD for \(E_{5,\sigma j}^{(n)}\), the integer
\(n\) is \(\theta\)-congruent if and only if the root number is \(-1\),
or it is \(+1\) and \(a_r(m)=0\).
\end{maintheorem}

\begin{proof}
The proof follows from Theorems~\ref{thm:explicit-central-values}
and~\ref{thm:tunnell-four-angles} in
Section~\ref{sec:tunnell-criteria}.
\end{proof}

\begin{maintheorem}[Non-\(\theta\)-congruent primes]
\label{thm:main-B}
A prime \(p\) is not \(\theta\)-congruent in the following cases:
\[
\begin{array}{ccl}
\cos\theta&M&p\pmod M\\
\hline
 3/5&40&11,19,21,29\\
 -3/5&40&3,7,23,27\\
 4/5&120&11,23,31,43,47,53,59,67,77,79\\
 -4/5&120&7,19,83,91,103,107.
\end{array}
\]
\end{maintheorem}

\begin{proof}
The proof follows from Theorem~\ref{thm:prime-non-theta-congruent} in
Section~\ref{sec:prime-noncongruence}.
\end{proof}

The paper is organized as follows.  In Section~\ref{sec:theta-congruent}, we recall Fujiwara's elliptic curve criterion and prove that every quadratic twist in the four families has torsion subgroup \((\mathbb Z/2\mathbb Z)^2\).
Thus \(\theta\)-congruence is equivalent to positive Mordell--Weil rank in these cases.  Section~\ref{sec:st-construction} specializes the Sirolli--Tornar\'ia construction to the classical setting over \(\mathbb Q\) and explains the ideal and lattice calculations used to obtain the local weights.  Section~\ref{sec:representative-40a1-computation} gives a complete representative calculation.
The resulting forms and coefficient formulas are summarized in Section~\ref{sec:explicit-series} and recorded in Appendices~\ref{app:theta-metadata}--\ref{app:fourier-data}.  In Section~\ref{sec:tunnell-criteria}, we give the Tunnell-type criterion for every positive square-free twisting parameter. Finally, Section~\ref{sec:prime-noncongruence} proves the unconditional prime non-\(\theta\)-congruence theorem using Sturm bounds and classical quadratic form theory, and concludes with brief computational details
and  elapsed times.

\section{\texorpdfstring{\(\theta\)-congruent numbers}{theta-congruent numbers}}
\label{sec:theta-congruent}

We retain the notation of the introduction.  For a nonzero square-free
integer \(d\), the \(d\)-th quadratic twist of \(E_{r,s}\) is
\[
  E_{r,s}^{(d)}:\quad
  y^2=x^3+2sdx^2-(r^2-s^2)d^2x
  =x\bigl(x-d(r-s)\bigr)\bigl(x+d(r+s)\bigr).
\]
Fujiwara's criterion takes the following form.

\begin{theorem}[Fujiwara \cite{Fu}]
\label{thm:fujiwara}
A positive square-free integer \(n\) is \(\theta\)-congruent if and only if
\(E_{r,s}^{(n)}(\mathbb Q)\) contains a point outside its rational
\(2\)-torsion subgroup.
\end{theorem}

For the four curves studied here, this criterion is equivalent to a rank
condition.

\begin{lemma}
\label{lem:torsion-5-3-5-4}
For every nonzero square-free integer \(d\), \(E_{5,\pm3}^{(d)}(\mathbb Q)_{\mathrm{tors}} \simeq E_{5,\pm4}^{(d)}(\mathbb Q)_{\mathrm{tors}} \simeq (\mathbb Z/2\mathbb Z)^2.\)
\end{lemma}

\begin{proof}
It suffices to treat \(s=3,4\), since \(E_{5,-s}^{(d)}=E_{5,s}^{(-d)}\).
Both curves have full rational \(2\)-torsion. 
Note that for a curve \(y^2=(x-e_1)(x-e_2)(x-e_3)\), a point $(e_i, 0)$ is divisible by $2$ over $\Q$ only if both $e_i - e_j$ and $e_i - e_k$ are rational squares. This excludes rational \(4\)-torsion: the
ratios of the two relevant root differences are \[-\frac14,\ \frac15,\ \frac45 \quad(s=3), \qquad -\frac19,\ \frac1{10},\ \frac9{10} \quad(s=4),\]
none of which is a square in \(\mathbb Q\).

A rational point of order \(3\) would, on substituting \(x=dz\) in
the \(3\)-division polynomial, give a rational root of one of the
following polynomials:
\[
 3z^4+24z^3-96z^2-256 \quad(s=3),
 \qquad
 3z^4+32z^3-54z^2-81 \quad(s=4).
\]
The first has no root modulo \(11\), and the second has no root modulo
\(19\). Any rational root has denominator dividing \(3\), so
reduction at these primes is defined. This excludes rational
\(3\)-torsion. Mazur's classification of rational torsion subgroups
\cite{Mazur} now gives the result.
\end{proof}

\begin{proposition}
\label{prop:rank-criterion-four-cases}
Let \(n\) be a positive square-free integer and let
\((r,s)\in\{(5,\pm3),(5,\pm4)\}\).  Then
\[
 n\text{ is \(\theta\)-congruent for }\cos\theta=s/r
 \quad\Longleftrightarrow\quad
 \operatorname{rank}E_{r,s}^{(n)}(\mathbb Q)>0.
\]
\end{proposition}

\begin{proof}
By Theorem~\ref{thm:fujiwara}, \(n\) is \(\theta\)-congruent exactly
when \(E_{r,s}^{(n)}(\mathbb Q)\) has a rational point outside its
\(2\)-torsion.  Lemma~\ref{lem:torsion-5-3-5-4} says that all rational
torsion is \(2\)-torsion, so this is equivalent to positive
Mordell--Weil rank of $E_{r, s}^{(n)}$ over \(\mathbb Q\).
\end{proof}

The next observation verifies the level hypothesis used in
Section~\ref{sec:st-construction}.

\begin{proposition}
\label{prop:conductor-odd-squarefree}
Let \(r,s\in\mathbb Z\) be coprime, with \(r>0\) and \(|s|<r\). The conductor of
\(E_{r,s}\) in \eqref{eq:thetacong} has the form \(2^kM\), where \(M\) is a positive odd
square-free integer.
\end{proposition}

\begin{proof}
The discriminant of the displayed model is
\(64r^2(r^2-s^2)^2\).  Thus an odd bad prime $p$ divides one of
\(r,r-s,r+s\).  Modulo $p$, the cubic
\(x(x-(r-s))(x+(r+s))\) has respectively the roots
\[
 0,-s,-s;\qquad 0,0,-2s;\qquad 0,-2s,0.
\]
Because \((r,s)=1\) and the prime is odd, each reduction has one double
root and one simple root.
Moreover, the $c_4$ invariant $16(3r^2 + s^2)$ of \eqref{eq:thetacong} is nonzero modulo $p$: if \(p\mid r\), then \(c_4\equiv16s^2\pmod p\), while if \(p\mid r-s\) or \(p\mid r+s\), then \(c_4\equiv64s^2\pmod p\). Hence the displayed model is \(p\)-minimal and has nodal reduction, so \(E_{r,s}\) has multiplicative reduction at \(p\). Therefore its conductor exponent at every odd bad prime is \(1\).
All other odd primes have good reduction, so the odd part of the conductor is square-free.
\end{proof}

\begin{rmk}
\label{rmk:nonzero-cosine-level}
The odd part in Proposition~\ref{prop:conductor-odd-squarefree} is
nontrivial whenever \(s\neq0\). Indeed, the proof shows that it is
\(M=\prod_{p\mid r(r-s)(r+s),\,p\text{ odd}}p\).
If \(M=1\), the positive integers \(r,r-s,r+s\) are powers of \(2\).
If \(r\) is odd, then \(r=1\), and \(|s|<r\) gives \(s=0\).
If \(r\) is even, coprimality makes \(s\) odd, so \(r-s\) and
\(r+s\) are odd. As powers of \(2\), they must both equal \(1\),
again giving \(s=0\). Thus \(M>1\) for every nonzero rational cosine.
The characters \(\chi_\delta\), \(\delta\in\{1,-1,2,-2\}\), are
unramified at odd primes, so all four base twists have the same odd
conductor part \(M\). In particular, their levels are nonsquares,
and the level hypothesis in Section~\ref{sec:st-construction}
is automatic.
\end{rmk}

\section{Central \texorpdfstring{\(L\)}{L}-values of quadratic twists}
\label{sec:st-construction}

In this section we specialize the generalized theta series construction of Sirolli--Tornar\'ia \cite{ST} to newforms of weight \(2\) over \(\mathbb Q\). Note that we use the classical normalization of the \(L\)-function, in which the central point for a form of weight \(2\) is \(s=1\). Thus the value written as \(L(1/2,f)\) in the normalization of \cite{ST} is denoted here by \(L(f,1)\).

\subsection{Preliminaries}
\label{subsec:st-preliminaries}

Let \(f(q)=\sum_{n\geq1}a_f(n)q^n\in S_2(\Gamma_0(N))\) be a normalized newform with trivial character. Every level used below has the form \(N=2^kM\), where \(M>1\) is odd and square-free. For \(t\in\mathbb Q^\times\), let \(\chi_t\) be the primitive quadratic character of its squareclass, trivial when \(t\) is a square. Here \(f\otimes\chi_t\) denotes the associated primitive newform twist. We include \(1\) among the fundamental discriminants. Put \(\Sigma_N=\{p:p\mid N\}\cup\{\infty\}\).  For \(p\mid N\), let \(\epsilon_f(p)\in\{\pm1\}\) be the Atkin--Lehner eigenvalue of \(f\), and set \(\epsilon_f(\infty)=-1\).

At the real place put \(\left(\frac D\infty\right)=\operatorname{sgn}(D)\). A \emph{type} is a function \(\gamma:\Sigma_N\to\{\pm1\}\).  A fundamental discriminant \(D\) is of type \(\gamma\) if, for every
\(v\in\Sigma_N\),
\begin{equation}
  \left(\frac{D}{v}\right)
  =
  \begin{cases}
    \gamma(p)\text{ or }0,
      &v=p\text{ is odd},\ p\parallel N,
       \text{ and }\gamma(p)=\epsilon_f(p),\\
    \gamma(v),&\text{otherwise}.
  \end{cases}
  \label{eq:type-definition}
\end{equation}
The root number is constant on each type and equals
\begin{equation}
  \epsilon_{f,\gamma}
  =\epsilon_f(\infty)\gamma(\infty)
   \prod_{p\mid N}\epsilon_f(p)
   \gamma(p)^{\operatorname{ord}_p(N)}.
  \label{eq:type-root-number}
\end{equation}
If \(\epsilon_{f,\gamma}=-1\), then \(L(f\otimes\chi_D,1)=0\) for every \(D\) of type \(\gamma\).  We therefore consider only types with \(\epsilon_{f,\gamma}=1\).

\begin{lemma}
\label{lem:four-base-twists-cover}
Let \(d\) be a nonzero square-free integer. Write \(d=\sigma 2^e m\), where \(\sigma\in\{1,-1\}\), \(e\in\{0,1\}\), and \(m\) is positive, odd, and square-free. Set
\[
  \eta_m=(-1)^{(m-1)/2},\qquad
  D_0=\eta_m m,\qquad
  \delta=\sigma 2^e\eta_m.
\]
Then \(D_0\) is an odd fundamental discriminant and $f\otimes \chi_d = f_\delta \otimes \chi_{D_0}$ where $f_\delta = f\otimes \chi_\delta$.
Also, $D_0$ is of type $\gamma$ with respect to $f_\delta$ for some type $\gamma$.
\end{lemma}

\begin{proof}
An odd square-free integer is a fundamental discriminant precisely when it is congruent to \(1\pmod 4\). By the definition of \(\eta_m\), we have \(D_0=\eta_m m\equiv1\pmod4\), so \(D_0\) is an odd fundamental discriminant. Moreover,
\(
  \delta D_0
  =\sigma2^e\eta_m^2m
  =\sigma2^em
  =d.
\)
It follows that
\(\chi_d=\chi_\delta\chi_{D_0}\), and hence
\(f\otimes\chi_d=f_\delta\otimes\chi_{D_0}\).
The odd part of the level of \(f_\delta\) is square-free. Since \(D_0\) is odd, its local symbols determine a type by \eqref{eq:type-definition}: at an odd bad prime dividing \(D_0\), choose the Atkin--Lehner sign, and use the corresponding symbol at every other place.
\end{proof}

Thus every quadratic twist of \(f\) can be written as an odd fundamental discriminant twist of one of the four fixed forms
\begin{equation}
    f_\delta, \quad \delta\in\{1,-1,2,-2\}.
    \label{eq:four-twists}
\end{equation}

\subsection{Quaternionic data and the Brandt module}
\label{subsec:st-brandt}

The construction of the generalized theta series begins by passing from the newform of weight \(2\) to its quaternionic counterpart via the Jacquet--Langlands correspondence \cite{JL}. We realize this correspondence explicitly using a Brandt module.
Fix one of the four base twists in \eqref{eq:four-twists}, denote it again by \(f\), and write its level as \(N=2^kM\), where \(M\) is odd and square-free. Fix a type \(\gamma\) with \(\epsilon_{f,\gamma}=1\).  
Choose an odd prime \(p_0\mid M\), let \(B/\mathbb Q\) be the definite quaternion algebra ramified exactly at \(p_0\) and \(\infty\), and let \(R\subset B\) be a Pizer--Eichler order of reduced discriminant \(N\), as in \cite[\S9]{ST}. For the families \eqref{eq:cos-values} considered in this paper, we take \(p_0=5\) when \(\cos\theta=\pm3/5\), and \(p_0=3\) when \(\cos\theta=\pm4/5\).

We use the standard Brandt module construction associated with \(R\).
See \cite[\S3]{ST} and \cite[pp.~116--139]{Gr}.
Let 
\begin{equation}\label{eq:ideal-classes}
    \operatorname{Cl}(R) = \{[I_1], \dots, [I_h]\}
\end{equation}
denote the set of locally principal right \(R\)-ideal classes. For each \(i\), let \(O_L(I_i)=\{b\in B:bI_i\subseteq I_i\}\) be the left order of \(I_i\), and put \(t_i=\#(O_L(I_i)^\times/\{\pm1\})\). The Brandt module \(\mathcal M(R)\) is the space of complex-valued functions on \(\operatorname{Cl}(R)\), or equivalently the complex vector space having basis \([I_1],\ldots,[I_h]\).
There is a pairing on $\mathcal{M}(R)$ given by
\begin{equation}
  \langle\varphi,\psi\rangle_B
  =
  \sum_{i=1}^h
  \frac{\varphi(I_i)\overline{\psi(I_i)}}{t_i}.
  \label{eq:brandt-pairing}
\end{equation}

Since \(f\) is a normalized newform, for each prime \(p\nmid N\) we can write $T_pf=a_f(p)f$ for the Hecke action. Let \(B(p)\) denote the Brandt matrix representing the action of \(T_p\) on \(\mathcal M(R)\).
By the Jacquet--Langlands correspondence, there is a common Hecke eigenline in \(\mathcal M(R)\) corresponding to \(f\), with the same eigenvalues \(a_f(p)\) for \(p\nmid N\). We therefore choose a nonzero vector \(\varphi_f\in\mathcal M(R)\) satisfying $B(p)\varphi_f=a_f(p)\varphi_f$ for $p \nmid N$.
In all cases considered here, finitely many of these equations determine a one-dimensional rational subspace. We choose its generator with primitive integral coordinates and write
\begin{equation*}\varphi_f=
    \left(\varphi_f(I_1),\ldots,\varphi_f(I_h)\right).
\end{equation*}

The order \(R\), the ideal-class representatives \(I_i\), and the Brandt matrices \(B(p)\) were computed using SageMath \cite{Sage}.
The vector \(\varphi_f\) was then obtained by solving the corresponding eigenspace equations over \(\mathbb Q\).

\subsection{The generalized theta series: structure and simplification}
\label{subsec:theta_series_reduction}

We first record the structure of the generalized theta series.
Choose an auxiliary parameter \(l\), either \(l=1\) or a fundamental discriminant with \(|l|\) an odd prime not dividing $N$, satisfying
\begin{equation}
  \gamma(\infty)l<0,\qquad
  \left(\frac{l}{p}\right)=
  \begin{cases}
    -\gamma(p),&p=p_0,\\
    \gamma(p),&p\mid N,\ p\neq p_0,
  \end{cases}
  \qquad
  L(f\otimes\chi_l,1)\neq0.
  \label{eq:auxiliary-local-conditions}
\end{equation}
Put \(\ell=|l|\). Such a choice is made for each of the finite types treated here; the general construction also allows composite auxiliary discriminants.
We also define
\[
  S_{\mathrm{II}}(\gamma)
  =
  \left\{
    q:q\text{ is prime},\ q\mid N,
    \gamma(q)^{\operatorname{ord}_q(N)}
    \neq\epsilon_f(q)
  \right\}.
\]
For each ideal class \([I_i]\) in \eqref{eq:ideal-classes}, Subsections~\ref{subsec:st-local-weights} and~\ref{subsec:typeII-computation} construct a positive-definite ternary quadratic form \(Q_i\), a first-kind weight function \(\Omega^{\mathrm{I}}_{i,\ell}\), and second-kind weight functions \(\Omega^{\mathrm{II}}_{i,q}\) for \(q\in S_{\mathrm{II}}(\gamma)\). The form \(Q_i\) and the weight functions are defined on \(\mathbb Z^3\). Next, put
\begin{equation}
\label{eq:weight-function}
    W_{i,\gamma}(v)
  =
  \Omega^{\mathrm{I}}_{i,\ell}(v)
  \prod_{q\in S_{\mathrm{II}}(\gamma)}
  \Omega^{\mathrm{II}}_{i,q}(v),
\end{equation}
where an empty product is understood to be \(1\), and define
\begin{equation*}
  \Theta_{i,\gamma}(q)
  =
  \sum_{n\geq0}
  \left(
    \sum_{\substack{v\in\mathbb Z^3\\Q_i(v)=\ell n}}
    W_{i,\gamma}(v)
  \right)q^n.
  \end{equation*}
The generalized theta series attached to \(f\) and \(\gamma\) is then defined by
\begin{equation}
  G_{f,\gamma}(q)
  =
  \sum_{i=1}^h
  \frac{\varphi_f(I_i)}{t_i}\,
  \Theta_{i,\gamma}(q).
  \label{eq:global-theta-shape}
\end{equation}
The individual constant terms are retained.  They vanish unless
\(l=1\) and \(S_{\mathrm{II}}(\gamma)=\varnothing\); in that case they
cancel in \(G_{f,\gamma}\), since the cuspidal Brandt vector satisfies
\(\sum_i\varphi_f(I_i)/t_i=0\).

After \(G_{f,\gamma}\) has been constructed, we simplify its expression by combining equivalent ideal-class contributions. Whenever two series $\Theta_{i, \gamma}(q)$ and $\Theta_{j,\gamma}(q)$ become identical after a unimodular change of variables, they are combined into one term. For the remaining terms, we compare their Fourier coefficients through the relevant Sturm bound. The comparisons include the constant term and use a common weight, level, and character. If two such series are proportional through this bound, then we retain a single representative and absorb the constant of proportionality into its coefficient. In this way, \(G_{f,\gamma}\) is expressed using one representative from each resulting proportionality class.

\subsection{Ternary lattices and the first-kind weight}
\label{subsec:st-local-weights}

We now construct the quadratic forms \(Q_i\) and the first-kind weights \(\Omega^{\mathrm{I}}_{i,\ell}\) occurring in \eqref{eq:weight-function}. We write \(\operatorname{trd}\) and \(\operatorname{nrd}\) for the reduced trace and reduced norm on \(B\), respectively. Thus, if \(x\mapsto\bar x\) denotes the standard involution of \(B\), then \(\operatorname{trd}(x)=x+\bar x\) and \(\operatorname{nrd}(x)=x\bar x\). Also, let \(\beta(X,Y)=\operatorname{nrd}(X+Y)-\operatorname{nrd}(X)-\operatorname{nrd}(Y)\) be the symmetric bilinear form associated with the reduced norm. For a \(\mathbb Z\)-lattice \(\Lambda\subset B\), write \(\Lambda^0=\{x\in\Lambda:\operatorname{trd}(x)=0\}\).

For each class \([I_i]\), choose an integral representative \(J_i\subset R\) of the same right ideal class.  When \(\ell>1\), we further require \(J_i\otimes_{\mathbb Z}\mathbb Z_\ell =R\otimes_{\mathbb Z}\mathbb Z_\ell\) and \(\left(\frac{N(J_i)}{\ell}\right)=1\), where \(N(J_i)\) is the positive generator of the reduced norm ideal of \(J_i\). Such representatives exist by weak approximation: use local principality and \(R\otimes\mathbb Z_\ell\simeq M_2(\mathbb Z_\ell)\) to impose these conditions at \(\ell\), then clear denominators by an integer prime to \(\ell\).
Let \(R_i=O_L(J_i)\) and \(L_i=(\mathbb Z+2R_i)^0\). \(L_i\) is the rank-three lattice over which the \(i\)th theta series is summed.
Choose a basis \(\mathcal E_i = (e_{i1},e_{i2},e_{i3})\) of \(L_i\), and use this basis for all quadratic forms and local weights below.
For \(v=(x,y,z)^t\in\mathbb Z^3\), put \(X_i(v)=xe_{i1}+ye_{i2}+ze_{i3}\) and \(Q_i(v)=\operatorname{nrd}(X_i(v))\).

We now define the first-kind weight function. If \(l=1\), the first-kind weight is trivial, and we set \(\Omega^{\mathrm{I}}_{i,1}(v)=1\). Assume henceforth that \(\ell>1\).  Fix a \(\mathbb Z\)-basis \(\mathcal E_R=(b_1,b_2,b_3)\) of \(R^0\).  By the normalization of \(J_i\), and since \(\ell\) is odd, \(L_i\otimes_{\mathbb Z}\mathbb Z_\ell =R^0\otimes_{\mathbb Z}\mathbb Z_\ell\).
Consequently, the coordinates \([e_{ij}]_{\mathcal E_R}\) lie in \(\mathbb Z_\ell^3\). Let \(C_i\in\operatorname{GL}_3(\mathbb F_\ell)\) be the reduction modulo \(\ell\) of the matrix 
\(\bigl([e_{i1}]_{\mathcal E_R}\ 
        [e_{i2}]_{\mathcal E_R}\ 
        [e_{i3}]_{\mathcal E_R}\bigr)\).
Therefore, if \(\bar v\in\mathbb F_\ell^3\) is the reduction of \(v\in\mathbb Z^3\), then \(C_i\bar v\) is the coordinate vector of the reduction of \(X_i(v)\) in $R^0/\ell R^0$ with basis induced from $\mathcal E_R$.
Choose \(b_0\in R^0\) whose reduction modulo \(\ell\) is nonzero and satisfies \(\operatorname{nrd}(b_0)\equiv0\pmod\ell\), and let \(c_0\in\mathbb F_\ell^3\) be the coordinate vector of this reduction with respect to \(\mathcal E_R\). Define \(S_i:\mathbb F_\ell^3\to\mathbb F_\ell\) by 
\[ S_i(x,y,z) = \overline{\beta(e_{i1},b_0)}x +\overline{\beta(e_{i2},b_0)}y +\overline{\beta(e_{i3},b_0)}z, 
\]
where the bars denote reduction modulo \(\ell\).
Then, the first-kind weight is defined by
\begin{equation}
  \Omega^{\mathrm{I}}_{i,\ell}(v)
  =
  \begin{cases}
    0,
      &\ell\nmid Q_i(v)\text{ or }\bar v=0,\\[1mm]
    \left(\dfrac{S_i(\bar v)}{\ell}\right),
      &\ell\mid Q_i(v),\ \bar v\neq0,
       \text{ and }S_i(\bar v)\neq0,\\[3mm]
    \left(\dfrac{k}{\ell}\right),
      &\ell\mid Q_i(v),\ \bar v\neq0,
       \text{ and }S_i(\bar v)=0,
  \end{cases}
  \label{eq:first-kind-coordinate-rule}
\end{equation}
where, in the last case, \(k\in\mathbb F_\ell^\times\) is the unique scalar satisfying \(C_i\bar v=kc_0\).
This is the first-kind weight of \cite[\S4.1, Proposition~4.2]{ST}, written in these coordinates.

\subsection{Second-kind weights}
\label{subsec:typeII-computation}

It remains to define the second-kind weights occurring in \eqref{eq:weight-function}.  For each odd prime \(q\in S_{\mathrm{II}}(\gamma)\), choose a Dirichlet character \(\psi_q\) of conductor \(q\) satisfying \(\psi_q(-1)=-1\). If \(2\in S_{\mathrm{II}}(\gamma)\), set \(\psi_2=\chi_{-4}\), the odd character of conductor \(4\). Extend each \(\psi_q\) by zero on nonunits, and put \(m_q=q\) for odd \(q\) and \(m_2=4\).

Let \(L=(\mathbb Z+2R)^0\). For \(q\in S_{\mathrm{II}}(\gamma)\), put \(B_q=B\otimes_{\mathbb Q}\mathbb Q_q\), and use the same notation \(\operatorname{trd}\), \(\operatorname{nrd}\), and \(\beta\) for their scalar extensions to \(B_q\). Choose \(z_q\in L\otimes_{\mathbb Z}\mathbb Z_q\) of unit reduced norm, using the local order descriptions and the second-kind construction in \cite[\S\S3--4]{ST}.
For each ideal class, local principality gives \(g_{i,q}\in B_q^\times\) such that \(J_i\otimes_{\mathbb Z}\mathbb Z_q =g_{i,q}(R\otimes_{\mathbb Z}\mathbb Z_q)\).
It follows that \(R_i\otimes_{\mathbb Z}\mathbb Z_q =g_{i,q}(R\otimes_{\mathbb Z}\mathbb Z_q)g_{i,q}^{-1}\), and hence \(z_{i,q}=g_{i,q}z_qg_{i,q}^{-1}\) belongs to \(L_i\otimes_{\mathbb Z}\mathbb Z_q\).
Extend \(X_i\) from Subsection~\ref{subsec:st-local-weights} \(\mathbb Z_q\)-linearly to \(\mathbb Z_q^3\), and define $\rho_{i, q}: \Z_q^3 \to \Z_q$ by
\begin{equation}
  \rho_{i,q}(v)
  =
  \frac{\beta(X_i(v),z_{i,q})}
       {2\operatorname{nrd}(z_{i,q})},
  \qquad v\in\mathbb Z_q^3.
  \label{eq:typeII-pairing}
\end{equation}
Since \(\beta(L_i,L_i)\subseteq2\mathbb Z\) and \(\operatorname{nrd}(z_{i,q})\) is a unit, \(\rho_{i,q}\) is \(\mathbb Z_q\)-valued. Let \(r_{i,q}:(\mathbb Z/m_q\mathbb Z)^3
\to\mathbb Z/m_q\mathbb Z\) be the linear form obtained by reducing \eqref{eq:typeII-pairing} modulo \(m_q\). The second-kind weight at \(q\) is defined by \(\Omega^{\mathrm{II}}_{i,q}(v) = \psi_q (r_{i,q}(v))\).

\subsection{The Sirolli--Tornar\'ia formula}
\label{subsec:st-theorem}

The construction of \(G_{f, \gamma}\) is now complete.  Write \(G_{f, \gamma}(q)=\sum_{m\geq1}a_\gamma(m)q^m\).  For an odd fundamental discriminant \(D\) of type \(\gamma\), the coefficient entering the central \(L\)-value formula is \(a_\gamma(|D|)\).
The following specializes \cite[Theorems~A--B, Propositions~2.3 and~5.1]{ST} to \(\mathbb Q\), with the holomorphic realization and the Petersson normalization used there.

\begin{theorem}[Sirolli--Tornar\'ia]
\label{thm:st-rational-specialization}
Let \(f\) and \(\gamma\) be as in Subsection~\ref{subsec:st-brandt}, and let \(l\) satisfy \eqref{eq:auxiliary-local-conditions}. Put \(\psi_\gamma=\prod_{q\in S_{\mathrm{II}}(\gamma)}\psi_q\).
Then the generalized theta series \(G_{f,\gamma}\) defined in \eqref{eq:global-theta-shape} is a nonzero holomorphic cusp form of weight \(3/2\) and level \(4\operatorname{lcm}\bigl(N,\operatorname{cond}(\psi_\gamma)^2\bigr)\). Its character is \(\psi_\gamma\chi_{-1}\) if \(\gamma(\infty)=1\), and \(\psi_\gamma\) if \(\gamma(\infty)=-1\).

For every odd fundamental discriminant \(D\) of type \(\gamma\),
\begin{equation}
  L(f\otimes\chi_l,1)L(f\otimes\chi_D,1)
  =
  2^{\omega(D,N)}c_{0,N}
  \frac{\langle f,f\rangle_{\mathrm{Pet}}}
       {\langle\varphi_f,\varphi_f\rangle_B}
  \frac{|a_\gamma(|D|)|^2}{\sqrt{|lD|}}.
  \label{eq:st-product-formula}
\end{equation}
Here,
\(\omega(D,N)=\#\{p:p\mid D\text{ and }p\mid N\}\),
\(\omega(N)=\#\{p:p\mid N\}\), and
\[
  c_{0,N}
  =
  2^{1-2\omega(N)}
  \prod_{p\mid N}(p+1)p^{\operatorname{ord}_p(N)-1}.
\]
Here \(l\) is the auxiliary parameter chosen in Subsection~\ref{subsec:theta_series_reduction};
\(\langle\cdot,\cdot\rangle_B\) is the pairing defined in
\eqref{eq:brandt-pairing}, and
\(\langle\cdot,\cdot\rangle_{\mathrm{Pet}}\) denotes the Petersson inner product in the normalization of \cite{ST}. The factor \(c_{0,N}\) is the specialization of \cite[(6.2)]{ST}; it must be changed if that inner product is rescaled.
Equivalently, \eqref{eq:st-product-formula} can be written as
\begin{equation*}
  L(f\otimes\chi_D,1)
  =
  2^{\omega(D,N)}
  \kappa(f,\gamma,l)
  \frac{|a_\gamma(|D|)|^2}{\sqrt{|D|}},
  \end{equation*}
where
\begin{equation}
\label{eq:st-final-formula-const}
  \kappa(f,\gamma,l)
  =
  \frac{c_{0,N}}
       {\sqrt{|l|}\,L(f\otimes\chi_l,1)}
  \frac{\langle f,f\rangle_{\mathrm{Pet}}}
       {\langle\varphi_f,\varphi_f\rangle_B}.
\end{equation}
In particular, \(L(f\otimes\chi_D,1)=0\) if and only if \(a_\gamma(|D|)=0\).
\end{theorem}

\section{A complete representative calculation:
the \texorpdfstring{\(40\mathrm{a}1\)}{40a1} case}
\label{sec:representative-40a1-computation}

We illustrate the construction of Section~\ref{sec:st-construction} in one complete example. Consider the base twist
\[
  E_{5,3}^{(-2)}:
  \quad
  y^2=x^3-12x^2-64x=x(x-16)(x+4),
\]
whose Cremona label is \(40\mathrm{a}1\). Let \(f=f_3^{(-2)}\) be its associated normalized newform, and choose the type given by $\gamma(2)=-1$, $\gamma(5)=-1$, and $\gamma(\infty)=1$.
The Atkin--Lehner signs are \(\epsilon_f(2)=1\) and \(\epsilon_f(5)=-1\). Since \(\operatorname{ord}_2(40)=3\), formula \eqref{eq:type-root-number} gives $\epsilon_{f, \gamma} = 1$.
Moreover, \(S_{\mathrm{II}}(\gamma)=\{2\}\).
We take \(l=-11\) and hence \(\ell=11\).

\subsection{Brandt and lattice data}

We use the quaternion algebra
\[
  B=\left(\frac{-2,-5}{\mathbb Q}\right)
  =\mathbb Q+\mathbb Qi+\mathbb Qj+\mathbb Qk,
  \qquad
  i^2=-2,\quad j^2=-5,\quad ij=k=-ji,
\]
and the Pizer--Eichler order
\(R=\langle\rho_1,\rho_2,\rho_3,\rho_4\rangle_{\mathbb Z}\), where
\[
  \rho_1=\frac{1+j+k}{2},\qquad
  \rho_2=\frac{i+5k}{2},\qquad
  \rho_3=j+k,\qquad
  \rho_4=4k.
\]
The Brandt module $\mathcal{M} (R)$ has dimension \(4\). The representatives of the right ideal classes used in the computation are
\begin{align*}
 J_1&=\langle4,\,16i,\,2+6i+2j,\,10i+2k\rangle_{\mathbb Z},\\
 J_2&=\langle12,\,48i,\,10+6i+2j,\,26i+2k\rangle_{\mathbb Z},\\
 J_3&=\langle12,\,4+16i,\,2+2i+6j,\,6+2j+2k\rangle_{\mathbb Z},\\
 J_4&=\langle12,\,8+16i,\,10+2i+6j,\,6i+4j+2k\rangle_{\mathbb Z}.
\end{align*}
Under the basis $\{[J_1], [J_2], [J_3], [J_4]\}$, we have $\varphi_f = (1, -1, 1, -1)$ and $t_i = 1$ for all $1 \le i \le 4$, and therefore \(\langle\varphi_f,\varphi_f\rangle_B=4\).

As in Subsection~\ref{subsec:st-local-weights}, let \(R_i=O_L(J_i)\) and \(L_i=(\mathbb Z+2R_i)^0\). For the subsequent computations, we use the following bases \(\mathcal E_i=(e_{i1},e_{i2},e_{i3})\) of \(L_i\):
\[
\begin{array}{c|ccc}
 i & e_{i1} & e_{i2} & e_{i3}\\
\hline
1
& j+k
& -3i-j
& -2i+j-k\\[1mm]
2
& -\frac23 i+j+\frac13 k
& -3i-j
& -\frac83 i+j-\frac53 k\\[1mm]
3
& -j-k
& 3i+\frac13 j-\frac23 k
& 2i-\frac53 j+\frac13 k\\[1mm]
4
& -i+\frac13 j+\frac23 k
& 2i-j+k
& 4i+\frac53 j+\frac13 k.
\end{array}
\]
With \(X_i(x,y,z)=xe_{i1}+ye_{i2}+ze_{i3}\), the corresponding quadratic forms \(Q_i=\operatorname{nrd}\circ X_i\) are
\begin{align*}
Q_1(x,y,z)
  &=15x^2+23y^2+23z^2-10xy-10xz+14yz,\\
Q_2(x,y,z)
  &=7x^2+23y^2+47z^2-2xy+6xz+22yz,\\
Q_3(x,y,z)
  &=15x^2+23y^2+23z^2+10xy+10xz+14yz,\\
Q_4(x,y,z)
  &=7x^2+23y^2+47z^2+2xy-6xz+22yz.
\end{align*}

\subsection{Local weights}
\label{subsec:local-weights}

For the first-kind weight, recall that \(l=-11\) and \(\ell=11\).
We use
\[
  \mathcal E_R=(b_1,b_2,b_3)
  =
  \left(
    \frac{i+5k}{2},\,j+k,\,4k
  \right)
\]
as a basis for \(R^0\). Choose 
\(b_0=-3j+k\in R^0\).  Its coordinate vector with respect to
\(\mathcal E_R\) is \((0,-3,1)^t\), so its reduction modulo \(11\) is
\[
  c_0=(0,8,1)^t\in\mathbb F_{11}^3.
\]
Moreover, \(\operatorname{nrd}(b_0)=55\equiv0\pmod{11}\), as required.
For \(\mathcal E_1=(j+k,-3i-j,-2i+j-k)\), the coordinate vectors of its basis elements with respect to \(\mathcal E_R\) are
\[
  (0,1,0)^t,\qquad
  (-6,-1,4)^t,\qquad
  (-4,1,2)^t.
\]
In this example, all of these coordinate vectors have integral entries; in general, we compute the coordinates over \(\mathbb Q\), regard them as elements of \(\mathbb Z_\ell^3\), and then reduce them modulo \(\ell\).
Reducing these columns modulo \(11\) gives \(C_1\).  Also,
\[
  \beta(e_{11},b_0)\equiv1,\qquad
  \beta(e_{12},b_0)\equiv8,\qquad
  \beta(e_{13},b_0)\equiv5
  \pmod{11},
\]
and hence \(S_1(x,y,z)=x+8y+5z\). The corresponding data for all four ideal classes are
\[
\begin{array}{c|c|c}
 i&C_i&S_i(x,y,z)\\
\hline
1&
\left(\begin{smallmatrix}
0&5&7\\
1&10&1\\
0&4&2
\end{smallmatrix}\right)
&x+8y+5z\\[3mm]
2&
\left(\begin{smallmatrix}
6&5&2\\
1&10&1\\
8&4&10
\end{smallmatrix}\right)
&6x+8y+10z\\[3mm]
3&
\left(\begin{smallmatrix}
0&6&4\\
10&4&2\\
0&7&9
\end{smallmatrix}\right)
&10x+6y+9z\\[3mm]
4&
\left(\begin{smallmatrix}
9&4&8\\
4&10&9\\
5&9&2
\end{smallmatrix}\right)
&7x+6y+8z
\end{array}
\qquad (\bmod\,11).
\]
The first-kind weight \(\Omega^{\mathrm{I}}_{i,11}\) is obtained from
\eqref{eq:first-kind-coordinate-rule}. In its last case, the second
coordinate gives \(k=7(C_i\bar v)_2\pmod{11}\), since \(8^{-1}=7\)
in \(\mathbb F_{11}\); the remaining coordinates verify
\(C_i\bar v=kc_0\).

For the second-kind weight, we have
\(S_{\mathrm{II}}(\gamma)=\{2\}\) and
\(\psi_2=\chi_{-4}\).  Put \(L=(\mathbb Z+2R)^0\) and choose
\(z_2=j+k=2\rho_1-1\in L\), for which
\(\operatorname{nrd}(z_2)=15\in\mathbb Z_2^\times\).
We choose \(g_{1,2}=4\) and \(g_{2,2}=g_{3,2}=g_{4,2}=12\).
Since these elements are central, conjugation fixes \(z_2\), so
\(z_{i,2}=z_2\) for every \(i\).  Its coordinate vectors with respect
to the chosen bases \(\mathcal E_i\) are
\[
\begin{aligned}
\mathbf z_{1,2}&=(1,0,0)^t,
&
\mathbf z_{2,2}&=\left(\frac43,0,-\frac13\right)^t,\\
\mathbf z_{3,2}&=(-1,0,0)^t,
&
\mathbf z_{4,2}&=\left(\frac43,0,\frac13\right)^t.
\end{aligned}
\]

Writing \(Q_i(v)=v^tA_iv\), the linear form
\(\rho_{i,2}\) in \eqref{eq:typeII-pairing} has coefficient vector
\(A_i\mathbf z_{i,2}/Q_i(\mathbf z_{i,2})\). For instance,
\[
  \frac{A_1\mathbf z_{1,2}}
       {Q_1(\mathbf z_{1,2})}
  =
  \begin{pmatrix}
    1\\ -1/3\\ -1/3
  \end{pmatrix}
  \equiv
  \begin{pmatrix}
    1\\1\\1
  \end{pmatrix}
  \pmod4.
\]
Reducing the corresponding linear forms modulo \(4\) gives
\[
\begin{aligned}
r_{1,2}(x,y,z)&=x+y+z,
&
r_{2,2}(x,y,z)&=x+y+z,\\
r_{3,2}(x,y,z)&=3x+y+z,
&
r_{4,2}(x,y,z)&=x+3y+3z
\end{aligned}
\qquad (\bmod\,4).
\]
Hence \(\Omega^{\mathrm{II}}_{i,2}(v)=\chi_{-4}(r_{i,2}(v))\), and the
total weight is
\(W_{i,\gamma}(v)=\Omega^{\mathrm{I}}_{i,11}(v)\chi_{-4}(r_{i,2}(v))\).

\subsection{Reduction of the formula}

We now combine the four ideal-class contributions to obtain a shorter representation formula for the Fourier coefficients.

For \(i=1,2,3,4\), let
\[
  \Theta_{i,\gamma}(q)
  =
  \sum_{n\geq0}
  \left(
    \sum_{\substack{v\in\mathbb Z^3\\Q_i(v)=11n}}
    W_{i,\gamma}(v)
  \right)q^n.
\]
Since
\(\varphi_f=(1,-1,1,-1)^t\) and \(t_i=1\) for all \(i\),
\(
  G_{f,\gamma}
  =
  \Theta_{1,\gamma}
  -\Theta_{2,\gamma}
  +\Theta_{3,\gamma}
  -\Theta_{4,\gamma}.
\)

To compare these four terms, consider the unimodular matrices
\[
  V_3=
  \begin{pmatrix}
    -1&0&0\\
    0&0&1\\
    0&1&0
  \end{pmatrix},
  \qquad
  V_4=
  \begin{pmatrix}
    -1&0&-1\\
    0&1&1\\
    0&0&-1
  \end{pmatrix}.
\]
They satisfy \(Q_3(V_3v)=Q_1(v)\) and \(Q_4(V_4v)=Q_2(v)\).
The same substitutions give \(W_{3,\gamma}(V_3v)=W_{1,\gamma}(v)\) and \(W_{4,\gamma}(V_4v)=W_{2,\gamma}(v)\). Since both matrices are unimodular, \(\Theta_{3,\gamma}=\Theta_{1,\gamma}\) and \(\Theta_{4,\gamma}=\Theta_{2,\gamma}\). Therefore, \(G_{f,\gamma}=2\Theta_{1,\gamma}-2\Theta_{2,\gamma}\).

The half-integral-weight level in this case is \(4\operatorname{lcm}(40,4^2)=320\). Hence the Sturm bound for weight \(3/2\) is
\[
  \left\lfloor
    \frac{3/2}{12}
    [\operatorname{SL}_2(\mathbb Z):\Gamma_0(320)]
  \right\rfloor
  =
  \frac18\cdot576
  =
  72.
\]
The Fourier coefficients of \(\Theta_{1,\gamma}\) and \(-\Theta_{2,\gamma}\) agree through \(q^{72}\). Sturm's theorem therefore gives \(\Theta_{1,\gamma}=-\Theta_{2,\gamma}\), and hence \(G_{f,\gamma}=4\Theta_{1,\gamma}\).

Writing \(G_{f,\gamma}(q)=\sum_{n\geq1}a_\gamma(n)q^n\), we obtain
\begin{equation*}
  a_\gamma(n)
  =
  4
  \sum_{\substack{(x,y,z)\in\mathbb Z^3\\Q_1(x,y,z)=11n}}
  \Omega^{\mathrm{I}}_{1,11}(x,y,z)\,
  \chi_{-4}(x+y+z),
  \end{equation*}
where
\[
  Q_1(x,y,z)
  =
  15x^2+23y^2+23z^2-10xy-10xz+14yz,
\]
and \(\Omega^{\mathrm{I}}_{1,11}\) is the first-kind weight determined by \(C_1\), \(S_1(x,y,z)=x+8y+5z\), and \(c_0=(0,8,1)^t\) as in Subsection~\ref{subsec:local-weights}. The initial Fourier expansion is
\[
\begin{aligned}
  G_{f,\gamma}(q)
  ={}&
  -8q^5+16q^{13}-16q^{37}+8q^{45}-16q^{53}
  +32q^{77}\\
  &+16q^{85}-32q^{93}+16q^{117}-8q^{125}
  +O(q^{151}).
\end{aligned}
\]

Finally, \(c_{0,40}=9\) and \(\langle\varphi_f,\varphi_f\rangle_B=4\).  
Since the auxiliary parameter is \(l=-11\), the constant in
\eqref{eq:st-final-formula-const} is
\[
  \kappa(f,\gamma,-11)
  =
  \frac{9}
       {4\sqrt{11}\,L(f\otimes\chi_{-11},1)}
  \langle f,f\rangle_{\mathrm{Pet}}.
\]
For every positive odd fundamental discriminant \(D\) of type \(\gamma\), equivalently satisfying
\[
  \left(\frac{D}{2}\right)=-1,
  \qquad
  \left(\frac{D}{5}\right)\in\{-1,0\},
\]
Theorem~\ref{thm:st-rational-specialization} gives
\[
  L(f\otimes\chi_D,1)
  =
  2^{\omega(D,40)}
  \kappa(f,\gamma,-11)
  \frac{|a_\gamma(D)|^2}{\sqrt D}.
\]

\section{Explicit generalized theta series}
\label{sec:explicit-series}

We apply Section~\ref{sec:st-construction} to
\(E_{5,3}:y^2=x(x-2)(x+8)\) and
\(E_{5,4}:y^2=x(x-1)(x+9)\).
Since \(E_{5,-j}=E_{5,j}^{(-1)}\), the negative-cosine cases belong to
these same twist families.  Recall that \(f_j^{(\mu)}\) is attached to
\(E_{5,j}^{(\mu)}\), for \(j\in\{3,4\}\) and
\(\mu\in\{1,-1,2,-2\}\), and that its level is \(N_j^{(\mu)}\).
Table~\ref{tab:base-twists} records these levels and the ramified primes.

We abbreviate \(G_{f_j^{(\mu)},\gamma}\) as
\(G_{3;\gamma_2,\gamma_5}^{(\mu)}\) or
\(G_{4;\gamma_2,\gamma_3,\gamma_5}^{(\mu)}\), where
\(\gamma_p=\gamma(p)\).  The finite signs and the base twist determine
\(\gamma(\infty)\) through \eqref{eq:type-root-number}, with root number
\(+1\).  Tables~\ref{tab:metadata-three} and~\ref{tab:metadata-four}
record the auxiliary parameter \(l\), Nebentypus, a level \(N_{3/2}\),
and the admissible classes of \(m=|D|\).
These classes select the type of an odd fundamental discriminant.
The heading \(d=m,2m,-m\), or \(-2m\) specifies the full twist
\(E_{5,j}^{(d)}\), whereas the superscript \((\mu)\) specifies the base
newform.  Section~\ref{sec:tunnell-criteria} gives the row-selection rule.

All coefficients lie in \(\mathbb Q(i)\), with rational coefficients
in some cases.  We use \(\psi_2=\chi_{-4}\),
\(\psi_3(a)=\left(\frac a3\right)\), and the character modulo \(5\)
defined by \(\psi_5(1)=1\), \(\psi_5(2)=i\),
\(\psi_5(3)=-i\), and \(\psi_5(4)=-1\).
Each character is periodic and zero on nonunits.

Tables~\ref{tab:three-fifths-summands} and
\ref{tab:four-fifths-summands} give the representation formulas.
In the \(Q\) column,
\(\left(\begin{smallmatrix}a&b&c\\r&s&t\end{smallmatrix}\right)\)
denotes
\(Q(x,y,z)=ax^2+by^2+cz^2+rxy+sxz+tyz\).
For \(p\in\{2,3,5\}\), an entry
\(\boldsymbol u_p=(u_{p,1},u_{p,2},u_{p,3})\) denotes the
second-kind factor
\[
  \psi_p(\boldsymbol u_p\cdot v),
  \qquad
  \boldsymbol u_p\cdot v
  =u_{p,1}x+u_{p,2}y+u_{p,3}z,
  \qquad
  v=(x,y,z)^t.
\]
An entry \(1\) means that the corresponding factor is absent.

The first-kind column gives \(\boldsymbol w\) above
\(\boldsymbol d\), where
\(S(v)=\boldsymbol w\cdot v\pmod\ell\) and
\(\boldsymbol d=C^{-1}c_0\).
Hence the last case of
\eqref{eq:first-kind-coordinate-rule} is equivalently
\(\bar v=k\boldsymbol d\).
Thus \(Q\), \(\ell=|l|\), \(\boldsymbol w\), and
\(\boldsymbol d\) determine the first-kind weight, which we denote by
\(\Omega^{\mathrm I}_{\ell;\boldsymbol w,\boldsymbol d}\).

A summand row contributes
\begin{equation*}
  c\!\sum_{\substack{v\in\mathbb Z^3\\Q(v)=\ell m}}
  \Omega^{\mathrm I}_{\ell;\boldsymbol w,\boldsymbol d}(v)
  \prod_{p\in\{2,3,5\}}
  \psi_p(\boldsymbol u_p\cdot v)
  \end{equation*}
to the coefficient of \(q^m\), with absent factors omitted.
The \(\psi_3\) factor is absent throughout
Table~\ref{tab:three-fifths-summands}.
All rows of a form block are added; a blank Form entry continues the
preceding block.

\begin{proposition}[Explicit generalized theta series]
\label{prop:computed-series}
Each row of Tables~\ref{tab:metadata-three} and~\ref{tab:metadata-four}
corresponds to the nonzero form \(G_{f_j^{(\mu)},\gamma}\) constructed
in Section~\ref{sec:st-construction}.  Its positive-index Fourier
coefficients are the sums in its block of
Table~\ref{tab:three-fifths-summands} or
Table~\ref{tab:four-fifths-summands}, and it has the listed level and
character.  The listed level need not be minimal.
\end{proposition}

\begin{proof}
Starting from \eqref{eq:global-theta-shape}, we apply the same
unimodular substitutions to each quadratic form and its weights.
We then combine proportional contributions by exact coefficient
comparison, including the constant term, through the Sturm bound in
their common modular-form space.
These operations preserve the series and its normalization.
Theorem~\ref{thm:st-rational-specialization} gives the level and
character.  Nonzero coefficients in Appendix~\ref{app:fourier-data}
also verify nonvanishing in every case.
\end{proof}

Appendix~\ref{app:fourier-data} records initial Fourier expansions of
the resulting forms.
The computation was implemented in SageMath~10.6 \cite{Sage}, using
PARI/GP~2.17.1 \cite{PARI} to enumerate lattice vectors. For each base
twist, the program computes the Brandt eigenvector, ternary lattices,
auxiliary discriminants, and local weights. It then evaluates the
weighted sums through their
Sturm bounds. The recorded running times, including the
coefficient comparisons of Section~\ref{sec:prime-noncongruence},
are given at the end of that section.

\section{Tunnell-type criteria}
\label{sec:tunnell-criteria}

Let \(\cos\theta=\sigma j/5\), with \(j\in\{3,4\}\) and
\(\sigma\in\{1,-1\}\), and write a positive square-free integer as
\(n=2^e m\), where \(e\in\{0,1\}\) and \(m\) is odd.
Put \(\eta_m=(-1)^{(m-1)/2}\), \(D_0=\eta_m m\),
\(\tau=\sigma2^e\), and \(\mu=\tau\eta_m\).
By Lemma~\ref{lem:four-base-twists-cover}, \(D_0\) is an odd
fundamental discriminant and
\begin{equation}
  \sigma n=\tau m=\mu D_0,\qquad
  L(E_{5,\sigma j}^{(n)},1)
  =L(f_j^{(\mu)}\otimes\chi_{D_0},1).
  \label{eq:twist-selection}
\end{equation}
The required coefficient is indexed by \(m\), even when \(n=2m\).

To select a row, use Table~\ref{tab:metadata-three} for \(j=3\) and
Table~\ref{tab:metadata-four} for \(j=4\).
In the block headed \(d=\tau m\), look for the superscript
\((\mu)\) and an admissible residue class containing \(m\).
There is exactly one such row when the root number is \(+1\), and
none when it is \(-1\).
The form label then identifies its entire summand block in
Appendix~\ref{app:weighted-formulas}.

Equivalently, determine the type from \eqref{eq:type-definition}:
put \(\gamma(\infty)=\operatorname{sgn}(D_0)\), and, for
\(p\mid N_j^{(\mu)}\), put \(\gamma(p)=\left(\frac{D_0}{p}\right)\)
when \(p\nmid D_0\).
At an odd bad prime dividing \(D_0\), put
\(\gamma(p)=\epsilon_{f_j^{(\mu)}}(p)\).
Thus a zero Kronecker symbol selects the Atkin--Lehner sign; it is
not itself a sign in a form label.
Equation~\eqref{eq:type-root-number} gives the root number.
For example, \(\cos\theta=3/5\) and \(n=22\) give
\(m=11\), \(D_0=-11\), \(\tau=2\), and \(\mu=-2\).
The appropriate row is \(G_{3;-,+}^{(-2)}\) in the \(d=2m\) block,
and its coefficient at \(q^{11}\) determines
\(L(E_{5,3}^{(22)},1)\).

\begin{theorem}[Explicit central \(L\)-value formulas]
\label{thm:explicit-central-values}
With the notation above, if the type of \(D_0\) with respect to
\(f_j^{(\mu)}\) has root number \(-1\), then
\(L(E_{5,\sigma j}^{(n)},1)=0\).
Otherwise, let \(r\) be the selected row, and write
\(G_r(q)=\sum_{a\geq1}a_r(a)q^a\).
There is an explicit constant \(\kappa_r>0\), depending only on the
row and its normalization, such that
\begin{equation}
  L(E_{5,\sigma j}^{(n)},1)
  =2^{\omega(D_0,N_j^{(\mu)})}\kappa_r
    \frac{|a_r(m)|^2}{\sqrt m}.
  \label{eq:explicit-row-formula}
\end{equation}
In particular, for root number \(+1\), the central \(L\)-value vanishes
if and only if \(a_r(m)=0\).
\end{theorem}

\begin{proof}
Use \eqref{eq:twist-selection} and the functional equation for root
number \(-1\).  For root number \(+1\), apply
Theorem~\ref{thm:st-rational-specialization} to the selected form and
use \(|D_0|=m\).
\end{proof}

\begin{rmk}
The constant \(\kappa_r\) is given by
\eqref{eq:st-final-formula-const}, with any overall rescaling of
\(G_r\) taken into account.  Its numerical value is unnecessary for
the vanishing criterion.
\end{rmk}

\begin{theorem}[Tunnell-type criterion for the four angles]
\label{thm:tunnell-four-angles}
Let \(n=2^e m\) be positive and square-free, and let
\(\cos\theta=\sigma j/5\), with \(j\in\{3,4\}\) and
\(\sigma\in\{1,-1\}\).
Determine the root number and, when it is \(+1\), the coefficient
\(a_r(m)\) by the rule above.
If the root number is \(+1\) and \(a_r(m)\neq0\), then \(n\) is not
\(\theta\)-congruent.
Assuming the rank part of the Birch--Swinnerton-Dyer conjecture for
\(E_{5,\sigma j}^{(n)}\), the integer \(n\) is \(\theta\)-congruent
if and only if either the root number is \(-1\), or it is \(+1\)
and \(a_r(m)=0\).
\end{theorem}

\begin{proof}
A nonzero \(a_r(m)\) gives \(L(E_{5,\sigma j}^{(n)},1)\neq0\).
Kolyvagin's theorem \cite{Ko} then gives Mordell--Weil rank zero,
and Proposition~\ref{prop:rank-criterion-four-cases} excludes
\(\theta\)-congruence.
Under BSD, vanishing of the central \(L\)-value is equivalent to positive
Mordell--Weil rank.  The functional equation and
\eqref{eq:explicit-row-formula} decide vanishing in the two root-number
cases, and the same proposition completes the proof.
\end{proof}

\section{Non-\texorpdfstring{\(\theta\)}{theta}-Congruent Primes}
\label{sec:prime-noncongruence}

We prove nonvanishing of the coefficients selected in
Section~\ref{sec:tunnell-criteria} for primes in certain residue classes.
Put \(\lambda=1+i\), so that
\(\mathbb Z[i]/(\lambda)\simeq\mathbb F_2\).
Let \(r_3(n)\) and \(r_{135}(n)\) denote the numbers of integral
representations of \(n\) by \(x^2+y^2+z^2\) and
\(x^2+3y^2+5z^2\), respectively, and write
\(\theta_Q(q)=\sum_{n\geq0}r_Q(n)q^n\).

For \(F(q)=\sum_{n\geq0}c_F(n)q^n\), put
\(\operatorname{pr}_{C,M}F
=\sum_{n\bmod M\in C}c_F(n)q^n\).
For the sets \(C\) occurring below, the indicator of \(C\) is a linear
combination of quadratic Dirichlet characters modulo \(M\), including
principal characters that vanish on nonunits. Thus the projection is
the corresponding linear combination of the coefficientwise twists
\(F\otimes\xi=\sum_{n\geq0}\xi(n)c_F(n)q^n\), which are modular
by the half-integral-weight twisting construction \cite{Sh}.

\subsection{Coefficient congruences}

For a row \(r\) of Table~\ref{tab:sturm-comparison-three} or
Table~\ref{tab:sturm-comparison-four}, write
\(G_r(q)=\sum_{n\geq1}a_r(n)q^n\), and let
\(C_r,e_r,Q_r,d_r\) be the data in that row.
Set \(M_r=40\) in the first table and \(M_r=120\) in the second.

The translation argument for the twisting construction \cite{Sh}
gives a sufficient common level \(L\) whenever
\(N\mid L\), \(4M^2\mid L\), and
\(M\operatorname{cond}(\chi)\mid L\) for every form being projected;
here \(N\) is its level and \(\chi\) its character. These conditions
hold with \(L_{40}=76800\) and \(L_{120}=460800\) for the forms in
the two tables and the comparison theta series, whose levels are
\(4\) and \(60\).
Put \(B_M=\frac18[\operatorname{SL}_2(\mathbb Z):\Gamma_0(L_M)]\).
The index formula \([\operatorname{SL}_2(\mathbb Z):\Gamma_0(L)]
=L\prod_{p\mid L}(1+1/p)\) gives
\(B_{40}=23040\) and \(B_{120}=138240\).

\begin{proposition}[Coefficient congruences]
\label{prop:sturm-certified-prime-congruences}
For every row \(r\) and every positive integer \(n\) with
\(n\bmod M_r\in C_r\),
\[
 \frac{a_r(n)}{2^{e_r}}\in\mathbb Z[i],\qquad
 \frac{r_{Q_r}(n)}{2^{d_r}}\in\mathbb Z,\qquad
 \frac{a_r(n)}{2^{e_r}}
 \equiv \frac{r_{Q_r}(n)}{2^{d_r}}\pmod\lambda .
\]
\end{proposition}

\begin{proof}
Fix \(r\), put \(M=M_r\), and set
\(F_r=\operatorname{pr}_{C_r,M}G_r\) and
\(H_r=\operatorname{pr}_{C_r,M}\theta_{Q_r}\).
These are forms of weight \(3/2\) at the common level \(L_M\), with
characters of order dividing \(4\).

We first record the finite congruence test used below.
If two such forms \(F,H\), with coefficients in \(\mathbb Z[i]\),
agree modulo \(\lambda\) for \(0\leq n\leq B_M\),
then \(F^4\) and \(H^4\) are forms of weight \(6\) and level \(L_M\)
with trivial character.  Since the residue field has characteristic
\(2\), agreement of \(F\) and \(H\) through \(B_M\) gives agreement
of their fourth powers through \(4B_M\), which is the weight \(6\)
Sturm bound.  Sturm's theorem \cite{Sturm} therefore gives
\(F^4\equiv H^4\pmod\lambda\), and hence
\(F\equiv H\pmod\lambda\), since fourth powers are injective in
\(\mathbb F_2[[q]]\).

Exact computation through \(B_M\) verifies the two divisibility
statements and the normalized congruence. Applying the test with
\(H=0\) first proves divisibility of every coefficient by \(\lambda\).
After division by \(\lambda\), the quotient still has integral
coefficients and lies in the same modular form space, so the test
applies again. Since \(2=-i\lambda^2\), a total of \(2e_r\) and
\(2d_r\) repetitions proves
\(F_r/2^{e_r}\in\mathbb Z[i][[q]]\) and
\(H_r/2^{d_r}\in\mathbb Z[[q]]\); the latter coefficients are integers
because they are also rational. Applying the test to these normalized
forms proves the congruence for all coefficients.
\end{proof}

\subsection{Parity of the comparison coefficients}

For a negative discriminant \(D\), let \(\operatorname{Cl}(D)\) be the
class group of primitive positive-definite binary quadratic forms of
discriminant \(D\), up to \(\operatorname{SL}_2(\mathbb Z)\)-equivalence.
Write \([A]\) for the class of \(A\), \(h(D)=\#\operatorname{Cl}(D)\),
and \(\operatorname{Cl}(D)^2\) for the subgroup of squares.

\begin{lemma}
\label{lem:three-square-prime-parity}
If \(p\equiv3\) or \(5\pmod8\) is prime, then \(r_3(p)/8\) is odd.
\end{lemma}

\begin{proof}
For square-free \(p>3\), Gauss's three-square formula gives
\(r_3(p)=24h(-p)\) when \(p\equiv3\pmod8\), and
\(r_3(p)=12h(-4p)\) when \(p\equiv5\pmod8\);
see \cite[\S1, (1)]{OnoSkinner}.
If \(p\equiv3\pmod8\), genus theory shows that \(h(-p)\) is odd, so
\(r_3(p)/8=3h(-p)\) is odd.  The case \(p=3\) is immediate from
\(r_3(3)=8\).

Now suppose \(p\equiv5\pmod8\), and let
\(P\) be the Sylow \(2\)-subgroup of \(\operatorname{Cl}(-4p)\).
Genus theory gives
\(\#(\operatorname{Cl}(-4p)/\operatorname{Cl}(-4p)^2)=2\);
hence \(P\) is nontrivial and cyclic.
Consider
\(A_p(x,y)=2x^2+2xy+(p+1)y^2/2\), of discriminant \(-4p\).
The inverse of the class of \([a,b,c]\) is represented by
\([a,-b,c]\), while the determinant \(1\) substitution
\((x,y)\mapsto(x-y,y)\) sends \(A_p\) to
\(2x^2-2xy+(p+1)y^2/2\).
Thus \([A_p]=[A_p]^{-1}\).

Let \(\xi_p\) be the genus character
\(\xi_p([A])=(\frac pt)\), where \(t\) is represented by \(A\) and
\((t,4p)=1\).  Since
\(A_p(1,1)=(p+9)/2\), quadratic reciprocity gives
\(\xi_p([A_p])=(\frac2p)=-1\).
Thus \([A_p]\) is a nonsquare element of order \(2\).
In a cyclic \(2\)-group of order at least \(4\), its unique element
of order \(2\) is a square; hence \(\#P=2\).
Therefore \(h(-4p)/2\) is odd, and so
\(r_3(p)/8=3h(-4p)/2\) is odd.
See \cite[\S3]{Cox} for the class group and genus theoretic facts.
\end{proof}

\begin{lemma}
\label{lem:135-prime-parity}
Let \(p\nmid30\) be prime with \(p\equiv7\pmod8\).
Then \(r_{135}(p)/4\) is an integer, odd precisely for
\(p\equiv7,23,31,47,79,103\pmod{120}\), and is even for
\(p\equiv71,119\pmod{120}\).
\end{lemma}

\begin{proof}
Reduction modulo \(8\) in
\(x^2+3y^2+5z^2=p\) forces \(y\) odd and \(x,z\) even.
Sign changes show that, modulo \(2\),
\(r_{135}(p)/4\) is the sum of the numbers \(A(p)\) and \(B(p)\) of
positive solutions of \(x^2+3y^2=p\) and \(3y^2+5z^2=p\),
respectively.

The only reduced primitive form of discriminant \(-12\) is
\(x^2+3y^2\).  By the representation criterion for reduced forms
\cite[Proposition~2.15]{Cox}, it represents \(p\) exactly when
\((\frac p3)=1\).  Since its genus contains one class,
\cite[Corollary~3.26]{Cox} shows that a represented prime has four
integral representations, hence one positive solution.  Thus
\(A(p)=(1+(\frac p3))/2\).

The primitive reduced forms of discriminant \(-60\) are
\(x^2+15y^2\) and \(3x^2+5y^2\).  Since both have zero middle
coefficient, \cite[Theorem~3.22]{Cox} shows that each genus of
discriminant \(-60\) consists of a single class.

If \(3y^2+5z^2=p\), reduction modulo \(3\) and \(5\) gives
\(\left(\frac p3\right)=\left(\frac p5\right)=-1\).
Conversely, suppose that both symbols are \(-1\).
Quadratic reciprocity gives
\(\left(\frac{-15}{p}\right)=\left(\frac p3\right)
\left(\frac p5\right)=1\).  Hence
\cite[Proposition~2.15]{Cox} shows that \(p\) is represented by one
of the two reduced forms of discriminant \(-60\).
The condition \(\left(\frac p3\right)=-1\) rules out
\(x^2+15y^2\), and therefore \(p=3y^2+5z^2\) for some integers
\(y,z\).

Since \(p\) is prime, \cite[Corollary~3.26]{Cox} gives exactly four
proper representations by the unique reduced form in this genus.
As \(p\nmid15\), neither coordinate can vanish, and these four
representations are precisely the sign changes of one positive
solution.  Thus
\(B(p)=(1-\left(\frac p3\right))(1-\left(\frac p5\right))/4\).
Thus \(A(p)+B(p)\) is odd exactly when \(\left(\frac p3\right)=1\),
or when both symbols are \(-1\). Together with \(p\equiv7\pmod8\),
the Chinese remainder theorem gives the stated residue classes.
\end{proof}

\subsection{The prime criterion}

\begin{theorem}[Non-\(\theta\)-congruent primes]
\label{thm:prime-non-theta-congruent}
A prime \(p\) is not \(\theta\)-congruent in each of the following cases:
\begin{center}
\begin{tabular}{ccl}
\toprule
\(\cos\theta\) & modulus & \(p\) \\
\midrule
\(3/5\)  & \(40\)  & \(11,19,21,29\pmod{40}\)\\
\(-3/5\) & \(40\)  & \(3,7,23,27\pmod{40}\)\\
\(4/5\)  & \(120\) & \(11,23,31,43,47,53,59,67,77,79\pmod{120}\)\\
\(-4/5\) & \(120\) & \(7,19,83,91,103,107\pmod{120}\)\\
\bottomrule
\end{tabular}
\end{center}
For \(3/5\), the first condition is equivalently
\((\frac p5)=1\) and \(p\equiv3,5\pmod8\);
for \(-3/5\), it is equivalently
\((\frac p5)=-1\) and \(p\equiv3,7\pmod8\).
\end{theorem}

\begin{proof}
For each listed prime class, choose the corresponding row of
Table~\ref{tab:sturm-comparison-three} or
Table~\ref{tab:sturm-comparison-four}.
Lemma~\ref{lem:three-square-prime-parity} or
Lemma~\ref{lem:135-prime-parity}, together with
Proposition~\ref{prop:sturm-certified-prime-congruences}, gives
\(a_r(p)/2^{e_r}\equiv1\pmod\lambda\).
Thus \(a_r(p)\neq0\), and
Theorem~\ref{thm:tunnell-four-angles} proves the assertion.
\end{proof}

\

\paragraph{\bf Computation times.}
The two families were computed in separate fresh runs on an Apple M1
MacBook Air with \(8\) GB of memory, under macOS~15.6.1, using
SageMath~10.6, Python~3.12.5, and PARI/GP~2.17.1.
The recorded elapsed times, in seconds, were
\begin{center}
\small
\begin{tabular}{@{}ccrrr@{}}
\toprule
\(\cos\theta\) & Forms & Construction & Congruences & Total\\
\midrule
\(\pm3/5\) & \(16\) & \(14.44\) & \(157.04\) & \(174.85\)\\
\(\pm4/5\) & \(32\) & \(1373.01\) & \(6323.89\) & \(7705.22\)\\
\bottomrule
\end{tabular}
\end{center}
Construction is the sum of the recorded curve, Brandt-data, and
complete-series stages, including Sturm simplification and the export
check. Congruences is the sum of the four or eight comparisons through
\(B_{40}=23040\) or \(B_{120}=138240\); all passed. Total includes
Sage startup and output generation. These are single-run measurements,
not averages.

\paragraph{\bf Computational assistance.}
The computations used throughout this paper, including the Brandt module
and lattice calculations, the construction and simplification of the
generalized theta series \(G_{f,\gamma}\), and the coefficient and Sturm
bound computations, were implemented with assistance from OpenAI's
GPT-6 Pro. ChatGPT also assisted with language editing and improving the clarity of the exposition.
The authors reviewed and independently verified the mathematical
arguments, the code used for the reported computations, and the
resulting data. The source code, input data, and coefficient files are
available from the authors upon request.

\appendix

\section{Levels, characters, and admissible indices}
\label{app:theta-metadata}

The notation and table conventions are given in
Section~\ref{sec:explicit-series}.  In Tables~\ref{tab:metadata-three}
and~\ref{tab:metadata-four}, \(m\) is positive, odd, and square-free,
and \(D=\eta_m m\), with \(\eta_m=(-1)^{(m-1)/2}\).
The admissible residue classes are determined by
\eqref{eq:type-definition}; they identify the appropriate form, even
when its coefficient at \(m\) vanishes.  Only root number \(+1\) types
are listed.  The rule for choosing a block and a row is in
Section~\ref{sec:tunnell-criteria}.

\begin{table}[htbp]
\centering
\small
\setlength{\tabcolsep}{10pt}
\begin{tabular}{@{}cccc@{}}
\toprule
Family \(j\) & Base twist \(\mu\) & \(N_j^{(\mu)}\) & Ramified prime \(p_0\)\\
\midrule
\(3\) & \(1,-1\) & \(320\) & \(5\)\\
\(3\) & \(2\) & \(80\) & \(5\)\\
\(3\) & \(-2\) & \(40\) & \(5\)\\
\addlinespace
\(4\) & \(1,-1\) & \(480\) & \(3\)\\
\(4\) & \(2,-2\) & \(960\) & \(3\)\\
\bottomrule
\end{tabular}
\caption{Base twists and ramified primes used in the construction.}
\label{tab:base-twists}
\end{table}

\begingroup
\scriptsize
\setlength{\LTleft}{0pt}
\setlength{\LTright}{\fill}
\setlength{\LTcapwidth}{\textwidth}
\setlength{\tabcolsep}{3pt}
\renewcommand{\arraystretch}{1.02}
\begin{longtable}{@{}>{\raggedright\arraybackslash}p{0.24\textwidth}>{\centering\arraybackslash}p{0.075\textwidth}>{\centering\arraybackslash}p{0.165\textwidth}>{\centering\arraybackslash}p{0.105\textwidth}>{\raggedright\arraybackslash}p{0.325\textwidth}@{}}
\caption{Generalized theta series for \(\cos\theta=\pm 3/5\).}\label{tab:metadata-three}\\
\toprule
Form & \(l\) & Character & \(N_{3/2}\) & Admissible \(m\pmod{40}\)\\
\midrule
\endfirsthead
\multicolumn{5}{@{}l}{\textit{Table~\thetable\ (continued)}}\\
\toprule
Form & \(l\) & Character & \(N_{3/2}\) & Admissible \(m\pmod{40}\)\\
\midrule
\endhead
\midrule
\multicolumn{5}{r@{}}{\textit{Continued on the next page}}\\
\endfoot
\bottomrule
\endlastfoot
\multicolumn{5}{@{}l}{\textit{Target twist }\(d=m\)}\\*
\(G_{3;-,+}^{(1)}\) & \(-43\) & \(1\) & \(1280\) & \(\{5,21,29\}\)\\
\(G_{3;+,+}^{(1)}\) & \(-23\) & \(1\) & \(1280\) & \(\{1,9,25\}\)\\
\(G_{3;-,+}^{(-1)}\) & \(13\) & \(1\) & \(1280\) & \(\{11,19,35\}\)\\
\(G_{3;+,+}^{(-1)}\) & \(17\) & \(1\) & \(1280\) & \(\{15,31,39\}\)\\
\addlinespace[0.4em]
\multicolumn{5}{@{}l}{\textit{Target twist }\(d=2m\)}\\*
\(G_{3;-,+}^{(-2)}\) & \(13\) & \(\psi_2\psi_5\) & \(1600\) & \(\{11,19\}\)\\
\(G_{3;+,-}^{(-2)}\) & \(1\) & \(1\) & \(160\) & \(\{7,15,23\}\)\\
\(G_{3;-,+}^{(2)}\) & \(-43\) & \(\psi_2\psi_5\) & \(1600\) & \(\{21,29\}\)\\
\(G_{3;+,+}^{(2)}\) & \(-7\) & \(\psi_2\psi_5\) & \(1600\) & \(\{1,9\}\)\\
\addlinespace[0.4em]
\multicolumn{5}{@{}l}{\textit{Target twist }\(d=-m\)}\\*
\(G_{3;-,-}^{(1)}\) & \(29\) & \(\psi_2\psi_5\) & \(6400\) & \(\{3,27\}\)\\
\(G_{3;+,-}^{(1)}\) & \(1\) & \(\psi_2\psi_5\) & \(6400\) & \(\{7,23\}\)\\
\(G_{3;-,-}^{(-1)}\) & \(-11\) & \(\psi_2\psi_5\) & \(6400\) & \(\{13,37\}\)\\
\(G_{3;+,-}^{(-1)}\) & \(-71\) & \(\psi_2\psi_5\) & \(6400\) & \(\{17,33\}\)\\
\addlinespace[0.4em]
\multicolumn{5}{@{}l}{\textit{Target twist }\(d=-2m\)}\\*
\(G_{3;-,-}^{(-2)}\) & \(-11\) & \(1\) & \(320\) & \(\{5,13,37\}\)\\
\(G_{3;+,+}^{(-2)}\) & \(-7\) & \(\psi_2\psi_5\) & \(800\) & \(\{1,9\}\)\\
\(G_{3;-,-}^{(2)}\) & \(29\) & \(1\) & \(320\) & \(\{3,27,35\}\)\\
\(G_{3;+,-}^{(2)}\) & \(1\) & \(1\) & \(320\) & \(\{7,15,23\}\)\\
\end{longtable}
\endgroup

\begingroup
\scriptsize
\setlength{\LTleft}{0pt}
\setlength{\LTright}{\fill}
\setlength{\LTcapwidth}{\textwidth}
\setlength{\tabcolsep}{3pt}
\renewcommand{\arraystretch}{1.02}
\begin{longtable}{@{}>{\raggedright\arraybackslash}p{0.24\textwidth}>{\centering\arraybackslash}p{0.075\textwidth}>{\centering\arraybackslash}p{0.165\textwidth}>{\centering\arraybackslash}p{0.105\textwidth}>{\raggedright\arraybackslash}p{0.325\textwidth}@{}}
\caption{Generalized theta series for \(\cos\theta=\pm 4/5\).}\label{tab:metadata-four}\\
\toprule
Form & \(l\) & Character & \(N_{3/2}\) & Admissible \(m\pmod{120}\)\\
\midrule
\endfirsthead
\multicolumn{5}{@{}l}{\textit{Table~\thetable\ (continued)}}\\
\toprule
Form & \(l\) & Character & \(N_{3/2}\) & Admissible \(m\pmod{120}\)\\
\midrule
\endhead
\midrule
\multicolumn{5}{r@{}}{\textit{Continued on the next page}}\\
\endfoot
\bottomrule
\endlastfoot
\multicolumn{5}{@{}l}{\textit{Target twist }\(d=m\)}\\*
\(G_{4;-,-,-}^{(1)}\) & \(-83\) & \(\psi_2\psi_3\) & \(5760\) & \(\{5,53,77\}\)\\
\(G_{4;-,+,+}^{(1)}\) & \(-19\) & \(\psi_2\psi_5\) & \(9600\) & \(\{21,61,69,109\}\)\\
\(G_{4;+,-,+}^{(1)}\) & \(-71\) & \(\psi_3\psi_5\) & \(28800\) & \(\{41,89\}\)\\
\(G_{4;+,+,-}^{(1)}\) & \(-103\) & \(1\) & \(1920\) & \(\{25,33,57,73,97,105\}\)\\
\(G_{4;-,-,-}^{(-1)}\) & \(37\) & \(1\) & \(1920\) & \(\{3,27,43,67,75,115\}\)\\
\(G_{4;-,+,+}^{(-1)}\) & \(29\) & \(\psi_3\psi_5\) & \(28800\) & \(\{11,59\}\)\\
\(G_{4;+,-,+}^{(-1)}\) & \(1\) & \(\psi_2\psi_5\) & \(9600\) & \(\{31,39,79,111\}\)\\
\(G_{4;+,+,-}^{(-1)}\) & \(113\) & \(\psi_2\psi_3\) & \(5760\) & \(\{23,47,95\}\)\\
\addlinespace[0.4em]
\multicolumn{5}{@{}l}{\textit{Target twist }\(d=2m\)}\\*
\(G_{4;-,-,+}^{(-2)}\) & \(61\) & \(\psi_2\psi_3\) & \(11520\) & \(\{19,91,115\}\)\\
\(G_{4;-,+,-}^{(-2)}\) & \(53\) & \(\psi_2\psi_5\) & \(19200\) & \(\{3,27,83,107\}\)\\
\(G_{4;+,-,+}^{(-2)}\) & \(1\) & \(\psi_2\psi_3\) & \(11520\) & \(\{31,55,79\}\)\\
\(G_{4;+,+,-}^{(-2)}\) & \(17\) & \(\psi_2\psi_5\) & \(19200\) & \(\{23,47,63,87\}\)\\
\(G_{4;-,-,+}^{(2)}\) & \(-11\) & \(1\) & \(3840\) & \(\{5,21,29,45,69,101\}\)\\
\(G_{4;-,+,-}^{(2)}\) & \(-43\) & \(\psi_3\psi_5\) & \(57600\) & \(\{13,37\}\)\\
\(G_{4;+,-,+}^{(2)}\) & \(-191\) & \(1\) & \(3840\) & \(\{9,41,65,81,89,105\}\)\\
\(G_{4;+,+,-}^{(2)}\) & \(-103\) & \(\psi_3\psi_5\) & \(57600\) & \(\{73,97\}\)\\
\addlinespace[0.4em]
\multicolumn{5}{@{}l}{\textit{Target twist }\(d=-m\)}\\*
\(G_{4;-,-,+}^{(1)}\) & \(61\) & \(\psi_3\psi_5\) & \(28800\) & \(\{19,91\}\)\\
\(G_{4;-,+,-}^{(1)}\) & \(53\) & \(1\) & \(1920\) & \(\{3,27,35,75,83,107\}\)\\
\(G_{4;+,-,-}^{(1)}\) & \(73\) & \(\psi_2\psi_3\) & \(5760\) & \(\{7,55,103\}\)\\
\(G_{4;+,+,+}^{(1)}\) & \(41\) & \(\psi_2\psi_5\) & \(9600\) & \(\{39,71,111,119\}\)\\
\(G_{4;-,-,+}^{(-1)}\) & \(-11\) & \(\psi_2\psi_5\) & \(9600\) & \(\{21,29,69,101\}\)\\
\(G_{4;-,+,-}^{(-1)}\) & \(-43\) & \(\psi_2\psi_3\) & \(5760\) & \(\{13,37,85\}\)\\
\(G_{4;+,-,-}^{(-1)}\) & \(-23\) & \(1\) & \(1920\) & \(\{17,33,57,65,105,113\}\)\\
\(G_{4;+,+,+}^{(-1)}\) & \(-31\) & \(\psi_3\psi_5\) & \(28800\) & \(\{1,49\}\)\\
\addlinespace[0.4em]
\multicolumn{5}{@{}l}{\textit{Target twist }\(d=-2m\)}\\*
\(G_{4;-,-,-}^{(-2)}\) & \(-83\) & \(\psi_3\psi_5\) & \(57600\) & \(\{53,77\}\)\\
\(G_{4;-,+,+}^{(-2)}\) & \(-19\) & \(1\) & \(3840\) & \(\{21,45,61,69,85,109\}\)\\
\(G_{4;+,-,-}^{(-2)}\) & \(-23\) & \(\psi_3\psi_5\) & \(57600\) & \(\{17,113\}\)\\
\(G_{4;+,+,+}^{(-2)}\) & \(-31\) & \(1\) & \(3840\) & \(\{1,9,25,49,81,105\}\)\\
\(G_{4;-,-,-}^{(2)}\) & \(13\) & \(\psi_2\psi_5\) & \(19200\) & \(\{3,27,43,67\}\)\\
\(G_{4;-,+,+}^{(2)}\) & \(29\) & \(\psi_2\psi_3\) & \(11520\) & \(\{11,35,59\}\)\\
\(G_{4;+,-,-}^{(2)}\) & \(73\) & \(\psi_2\psi_5\) & \(19200\) & \(\{7,63,87,103\}\)\\
\(G_{4;+,+,+}^{(2)}\) & \(449\) & \(\psi_2\psi_3\) & \(11520\) & \(\{71,95,119\}\)\\
\end{longtable}
\endgroup

\section{Weighted representation formulas}
\label{app:weighted-formulas}

Each block below gives the summands of one form, with
\(\ell=|l|\) taken from the corresponding metadata row.
The \(Q\) column lists its coefficients in two rows, as defined in
Section~\ref{sec:explicit-series}.  The first-kind column likewise
lists \(\boldsymbol w\) above \(\boldsymbol d=C^{-1}c_0\).
The vector \(c_0\) is recorded under the form label, once per block;
the case \(l=1\) has no such vector.
The columns headed \(\boldsymbol u_p\) specify
\(\psi_p(\boldsymbol u_p\cdot v)\), and \(1\) means an absent factor.
Thus every summand is determined by the printed data, including the
case in which the first-kind linear form vanishes.

\begingroup
\footnotesize
\renewcommand{\qfdata}[2]{\ensuremath{\left(\begin{smallmatrix}#1\\#2\end{smallmatrix}\right)}}
\renewcommand{\fkdata}[2]{\ensuremath{\left(\begin{smallmatrix}#1\\#2\end{smallmatrix}\right)}}
\setlength{\LTleft}{0pt}
\setlength{\LTright}{\fill}
\setlength{\LTcapwidth}{\textwidth}
\setlength{\tabcolsep}{1.7pt}
\setlength{\arraycolsep}{3pt}
\renewcommand{\arraystretch}{1.02}
\begin{longtable}{@{}>{\raggedright\arraybackslash}p{0.235\textwidth}>{\centering\arraybackslash}p{0.045\textwidth}>{\centering\arraybackslash}p{0.185\textwidth}>{\centering\arraybackslash}p{0.195\textwidth}>{\centering\arraybackslash}p{0.115\textwidth}>{\raggedright\arraybackslash}p{0.115\textwidth}@{}}
\caption{Weighted representation sums for \(\cos\theta=\pm 3/5\).}\label{tab:three-fifths-summands}\\
\toprule
Form & \(c\) & \(Q\) & \shortstack{First-kind data\\\(\boldsymbol w\) above \(\boldsymbol d\)} & \(\boldsymbol u_2\) & \(\boldsymbol u_5\)\\
\midrule
\endfirsthead
\multicolumn{6}{@{}l}{\textit{Table~\thetable\ (continued)}}\\
\toprule
Form & \(c\) & \(Q\) & \shortstack{First-kind data\\\(\boldsymbol w\) above \(\boldsymbol d\)} & \(\boldsymbol u_2\) & \(\boldsymbol u_5\)\\
\midrule
\endhead
\midrule
\multicolumn{6}{r@{}}{\textit{Continued on the next page}}\\
\endfoot
\bottomrule
\endlastfoot
\(G_{3;-,+}^{(1)}\)\newline \(c_0=(28,4,1)^t\) & \(2\) & \qfdata{15&87&343}{10&10&-82} & \fkdata{10&5&0}{22&42&22} & \vct{1,3,3} & 1\\*[2pt]
 & \(2\) & \qfdata{28&47&327}{-12&-28&6} & \fkdata{26&17&6}{5&20&22} & \vct{2,3,3} & 1\\*[2pt]
 & \(-6\) & \qfdata{28&95&183}{20&4&-90} & \fkdata{17&16&31}{23&42&42} & \vct{2,1,1} & 1\\*[2pt]
 & \(-4\) & \qfdata{60&87&92}{20&-40&36} & \fkdata{10&8&3}{36&22&22} & \vct{2,1,2} & 1\\*[2pt]
 & \(-2\) & \qfdata{23&60&335}{-20&10&-60} & \fkdata{17&26&6}{22&7&22} & \vct{3,2,1} & 1\\
\addlinespace[0.3em]
\(G_{3;+,+}^{(1)}\)\newline \(c_0=(12,1,1)^t\) & \(6\) & \qfdata{15&87&343}{10&10&-82} & \fkdata{5&7&20}{15&1&12} & \vct{1,3,3} & 1\\*[2pt]
 & \(2\) & \qfdata{28&47&327}{-12&-28&6} & \fkdata{15&8&2}{15&12&12} & \vct{2,3,3} & 1\\*[2pt]
 & \(-4\) & \qfdata{28&95&183}{20&4&-90} & \fkdata{16&12&18}{22&1&13} & \vct{2,1,1} & 1\\*[2pt]
 & \(-2\) & \qfdata{7&183&367}{2&-6&182} & \fkdata{0&22&9}{16&16&12} & \vct{3,1,1} & 1\\*[2pt]
 & \(-2\) & \qfdata{23&60&335}{-20&10&-60} & \fkdata{8&15&6}{21&12&11} & \vct{3,2,1} & 1\\
\addlinespace[0.3em]
\(G_{3;-,+}^{(-1)}\)\newline \(c_0=(8,7,1)^t\) & \(8\) & \qfdata{47&92&112}{44&24&-16} & \fkdata{3&7&10}{0&9&8} & 1 & 1\\*[2pt]
 & \(-8\) & \qfdata{23&112&167}{8&-2&-56} & \fkdata{1&11&10}{5&6&2} & 1 & 1\\
\addlinespace[0.3em]
\(G_{3;+,+}^{(-1)}\)\newline \(c_0=(9,10,1)^t\) & \(8\) & \qfdata{47&92&112}{-44&24&16} & \fkdata{14&16&3}{2&9&5} & 1 & 1\\*[2pt]
 & \(8\) & \qfdata{47&55&167}{10&26&30} & \fkdata{2&4&5}{14&15&13} & 1 & 1\\
\addlinespace[0.3em]
\(G_{3;-,+}^{(-2)}\)\newline \(c_0=(10,3,1)^t\) & \(4\) & \qfdata{15&23&23}{-10&-10&14} & \fkdata{7&1&0}{1&6&8} & \vct{1,1,1} & \vct{0,2,3}\\
\addlinespace[0.3em]
\(G_{3;+,-}^{(-2)}\) & \(2\) & \qfdata{15&23&23}{-10&-10&14} & 1 & 1 & 1\\*[2pt]
 & \(-2\) & \qfdata{7&23&47}{-2&6&22} & 1 & 1 & 1\\
\addlinespace[0.3em]
\(G_{3;-,+}^{(2)}\)\newline \(c_0=(40,3,1)^t\) & \(4\) & \qfdata{15&23&87}{-10&10&18} & \fkdata{14&31&1}{25&22&0} & 1 & \vct{0,2,1}\\*[2pt]
 & \(-4\) & \qfdata{23&28&47}{4&22&-12} & \fkdata{12&0&42}{21&31&37} & 1 & \vct{2,3,4}\\
\addlinespace[0.3em]
\(G_{3;+,+}^{(2)}\)\newline \(c_0=(4,0,1)^t\) & \(2\) & \qfdata{15&23&87}{-10&10&18} & \fkdata{3&4&0}{4&4&0} & 1 & \vct{0,2,1}\\*[2pt]
 & \(-4\) & \qfdata{23&28&47}{-4&-22&-12} & \fkdata{3&0&2}{0&1&0} & 1 & \vct{2,2,1}\\*[2pt]
 & \(2\) & \qfdata{7&47&92}{6&4&-44} & \fkdata{4&4&5}{4&4&2} & 1 & \vct{1,4,1}\\
\addlinespace[0.3em]
\(G_{3;-,-}^{(1)}\)\newline \(c_0=(16,0,1)^t\) & \(2\) & \qfdata{15&87&343}{10&10&-82} & \fkdata{19&12&19}{14&0&15} & \vct{1,3,3} & \vct{0,1,2}\\*[2pt]
 & \(2\) & \qfdata{28&47&327}{12&28&6} & \fkdata{13&25&22}{3&27&15} & \vct{2,3,3} & \vct{3,1,4}\\*[2pt]
 & \(-6\) & \qfdata{28&95&183}{20&4&-90} & \fkdata{6&12&19}{15&0&9} & \vct{2,1,1} & \vct{3,0,2}\\*[2pt]
 & \(-4\) & \qfdata{60&87&92}{-20&40&36} & \fkdata{16&4&16}{9&8&18} & \vct{2,3,2} & \vct{0,4,1}\\*[2pt]
 & \(-2\) & \qfdata{23&60&335}{20&10&60} & \fkdata{19&2&0}{24&4&14} & \vct{1,2,3} & \vct{2,0,0}\\
\addlinespace[0.3em]
\(G_{3;+,-}^{(1)}\) & \(6\) & \qfdata{15&87&343}{10&10&-82} & 1 & \vct{1,3,3} & \vct{0,1,2}\\*[2pt]
 & \(2\) & \qfdata{28&47&327}{-12&-28&6} & 1 & \vct{2,3,3} & \vct{2,1,4}\\*[2pt]
 & \(-4\) & \qfdata{28&95&183}{20&4&-90} & 1 & \vct{2,1,1} & \vct{3,0,2}\\*[2pt]
 & \(-2\) & \qfdata{7&183&367}{2&-6&182} & 1 & \vct{3,1,1} & \vct{1,3,1}\\*[2pt]
 & \(-2\) & \qfdata{23&60&335}{-20&10&-60} & 1 & \vct{3,2,1} & \vct{3,0,0}\\
\addlinespace[0.3em]
\(G_{3;-,-}^{(-1)}\)\newline \(c_0=(7,8,1)^t\) & \(8\) & \qfdata{47&92&112}{44&24&-16} & \fkdata{5&3&0}{6&1&1} & 1 & \vct{4,4,4}\\*[2pt]
 & \(-8\) & \qfdata{23&112&167}{8&-2&-56} & \fkdata{4&5&1}{2&4&5} & 1 & \vct{2,1,1}\\
\addlinespace[0.3em]
\(G_{3;+,-}^{(-1)}\)\newline \(c_0=(53,3,1)^t\) & \(8\) & \qfdata{47&92&112}{44&24&-16} & \fkdata{29&1&42}{46&38&62} & 1 & \vct{4,4,4}\\*[2pt]
 & \(8\) & \qfdata{47&55&167}{10&26&30} & \fkdata{42&48&14}{50&47&54} & 1 & \vct{1,0,4}\\
\addlinespace[0.3em]
\(G_{3;-,-}^{(-2)}\)\newline \(c_0=(0,8,1)^t\) & \(4\) & \qfdata{15&23&23}{-10&-10&14} & \fkdata{1&8&5}{5&1&4} & \vct{1,1,1} & 1\\
\addlinespace[0.3em]
\(G_{3;+,+}^{(-2)}\)\newline \(c_0=(5,1,1)^t\) & \(2\) & \qfdata{15&23&23}{-10&-10&14} & \fkdata{3&2&4}{4&0&4} & 1 & \vct{0,2,3}\\*[2pt]
 & \(-2\) & \qfdata{7&23&47}{2&-6&22} & \fkdata{1&5&6}{1&1&6} & 1 & \vct{4,2,4}\\
\addlinespace[0.3em]
\(G_{3;-,-}^{(2)}\)\newline \(c_0=(26,2,1)^t\) & \(8\) & \qfdata{15&23&87}{-10&10&18} & \fkdata{12&27&5}{17&15&0} & 1 & 1\\
\addlinespace[0.3em]
\(G_{3;+,-}^{(2)}\) & \(2\) & \qfdata{15&23&87}{-10&10&18} & 1 & 1 & 1\\*[2pt]
 & \(-4\) & \qfdata{23&28&47}{-4&-22&-12} & 1 & 1 & 1\\*[2pt]
 & \(2\) & \qfdata{7&47&92}{-6&4&44} & 1 & 1 & 1\\
\addlinespace[0.3em]
\end{longtable}
\endgroup

\begingroup
\footnotesize
\renewcommand{\qfdata}[2]{\ensuremath{\left(\begin{smallmatrix}#1\\#2\end{smallmatrix}\right)}}
\renewcommand{\fkdata}[2]{\ensuremath{\left(\begin{smallmatrix}#1\\#2\end{smallmatrix}\right)}}
\setlength{\LTleft}{0pt}
\setlength{\LTright}{\fill}
\setlength{\LTcapwidth}{\textwidth}
\setlength{\tabcolsep}{1.7pt}
\setlength{\arraycolsep}{3pt}
\renewcommand{\arraystretch}{1.02}
\begin{longtable}{@{}>{\raggedright\arraybackslash}p{0.205\textwidth}>{\centering\arraybackslash}p{0.04\textwidth}>{\centering\arraybackslash}p{0.175\textwidth}>{\centering\arraybackslash}p{0.18\textwidth}>{\centering\arraybackslash}p{0.105\textwidth}>{\centering\arraybackslash}p{0.105\textwidth}>{\raggedright\arraybackslash}p{0.105\textwidth}@{}}
\caption{Weighted representation sums for \(\cos\theta=\pm 4/5\).}\label{tab:four-fifths-summands}\\
\toprule
Form & \(c\) & \(Q\) & \shortstack{First-kind data\\\(\boldsymbol w\) above \(\boldsymbol d\)} & \(\boldsymbol u_2\) & \(\boldsymbol u_3\) & \(\boldsymbol u_5\)\\
\midrule
\endfirsthead
\multicolumn{7}{@{}l}{\textit{Table~\thetable\ (continued)}}\\
\toprule
Form & \(c\) & \(Q\) & \shortstack{First-kind data\\\(\boldsymbol w\) above \(\boldsymbol d\)} & \(\boldsymbol u_2\) & \(\boldsymbol u_3\) & \(\boldsymbol u_5\)\\
\midrule
\endhead
\midrule
\multicolumn{7}{r@{}}{\textit{Continued on the next page}}\\
\endfoot
\bottomrule
\endlastfoot
\(G_{4;-,-,-}^{(1)}\)\newline \(c_0=(76,6,1)^t\) & \(2\) & \qfdata{60&111&175}{-60&60&-30} & \fkdata{1&42&78}{60&16&80} & 1 & \vct{0,0,1} & 1\\*[2pt]
 & \(-4\) & \qfdata{55&151&151}{50&-50&82} & \fkdata{48&60&27}{81&22&10} & 1 & \vct{1,1,2} & 1\\*[2pt]
 & \(4\) & \qfdata{60&79&271}{-60&60&34} & \fkdata{74&58&36}{80&78&5} & 1 & \vct{0,2,1} & 1\\*[2pt]
 & \(-4\) & \qfdata{39&156&199}{-36&18&-156} & \fkdata{29&73&41}{13&42&80} & 1 & \vct{0,0,1} & 1\\*[2pt]
 & \(-2\) & \qfdata{15&256&256}{0&0&128} & \fkdata{60&31&4}{2&26&59} & 1 & \vct{0,2,2} & 1\\
\addlinespace[0.3em]
\(G_{4;-,+,+}^{(1)}\)\newline \(c_0=(9,12,1)^t\) & \(2\) & \qfdata{60&111&175}{-60&60&-30} & \fkdata{2&9&1}{6&5&0} & 1 & 1 & \vct{0,2,0}\\*[2pt]
 & \(-4\) & \qfdata{55&151&151}{50&-50&82} & \fkdata{6&16&5}{1&7&3} & 1 & 1 & \vct{0,3,3}\\*[2pt]
 & \(4\) & \qfdata{60&79&271}{-60&60&34} & \fkdata{14&7&15}{3&18&4} & 1 & 1 & \vct{0,4,2}\\*[2pt]
 & \(-4\) & \qfdata{39&156&199}{-36&18&-156} & \fkdata{8&2&18}{5&18&0} & 1 & 1 & \vct{1,3,1}\\*[2pt]
 & \(-2\) & \qfdata{15&256&256}{0&0&128} & \fkdata{9&18&13}{6&2&15} & 1 & 1 & \vct{0,3,2}\\
\addlinespace[0.3em]
\(G_{4;+,-,+}^{(1)}\)\newline \(c_0=(68,1,1)^t\) & \(16\) & \qfdata{60&111&175}{-60&60&-30} & \fkdata{16&36&63}{9&52&39} & \vct{2,1,3} & \vct{0,0,1} & \vct{0,2,0}\\
\addlinespace[0.3em]
\(G_{4;+,+,-}^{(1)}\)\newline \(c_0=(101,1,1)^t\) & \(2\) & \qfdata{60&111&175}{-60&60&-30} & \fkdata{49&4&48}{91&9&63} & \vct{2,3,1} & 1 & 1\\*[2pt]
 & \(-4\) & \qfdata{55&151&151}{50&-50&82} & \fkdata{96&60&73}{72&95&29} & \vct{3,1,3} & 1 & 1\\*[2pt]
 & \(6\) & \qfdata{60&79&271}{-60&60&34} & \fkdata{41&60&29}{97&77&9} & \vct{2,3,1} & 1 & 1\\*[2pt]
 & \(-4\) & \qfdata{39&156&199}{-36&18&-156} & \fkdata{98&97&78}{4&26&63} & \vct{3,2,1} & 1 & 1\\
\addlinespace[0.3em]
\(G_{4;-,-,-}^{(-1)}\)\newline \(c_0=(28,3,1)^t\) & \(8\) & \qfdata{31&124&255}{4&-30&60} & \fkdata{30&15&31}{35&13&4} & 1 & 1 & 1\\*[2pt]
 & \(8\) & \qfdata{79&111&124}{-66&-28&36} & \fkdata{12&3&19}{15&13&6} & 1 & 1 & 1\\
\addlinespace[0.3em]
\(G_{4;-,+,+}^{(-1)}\)\newline \(c_0=(23,1,1)^t\) & \(8\) & \qfdata{31&124&255}{4&-30&60} & \fkdata{12&20&4}{4&24&13} & 1 & \vct{2,1,0} & \vct{3,1,0}\\*[2pt]
 & \(8\) & \qfdata{79&111&124}{-66&-28&36} & \fkdata{14&13&19}{1&23&11} & 1 & \vct{1,0,1} & \vct{4,2,1}\\
\addlinespace[0.3em]
\(G_{4;+,-,+}^{(-1)}\) & \(8\) & \qfdata{31&124&255}{4&-30&60} & 1 & \vct{1,2,3} & 1 & \vct{3,1,0}\\*[2pt]
 & \(8\) & \qfdata{79&111&124}{-66&-28&36} & 1 & \vct{3,3,2} & 1 & \vct{4,2,1}\\
\addlinespace[0.3em]
\(G_{4;+,+,-}^{(-1)}\)\newline \(c_0=(82,8,1)^t\) & \(8\) & \qfdata{31&124&255}{4&-30&60} & \fkdata{92&67&70}{34&78&55} & \vct{1,2,3} & \vct{2,1,0} & 1\\*[2pt]
 & \(8\) & \qfdata{79&111&124}{-66&-28&36} & \fkdata{14&104&81}{98&2&7} & \vct{3,3,2} & \vct{1,0,1} & 1\\
\addlinespace[0.3em]
\(G_{4;-,-,+}^{(-2)}\)\newline \(c_0=(36,10,1)^t\) & \(16\) & \qfdata{31&255&496}{30&8&-120} & \fkdata{27&53&16}{49&52&31} & \vct{1,1,0} & \vct{2,0,2} & 1\\*[2pt]
 & \(-16\) & \qfdata{124&175&256}{-100&64&160} & \fkdata{28&39&4}{21&10&30} & \vct{2,1,0} & \vct{2,2,1} & 1\\*[2pt]
 & \(32\) & \qfdata{151&156&220}{-132&-100&120} & \fkdata{60&8&20}{37&20&3} & \vct{3,2,2} & \vct{2,0,2} & 1\\
\addlinespace[0.3em]
\(G_{4;-,+,-}^{(-2)}\)\newline \(c_0=(38,5,1)^t\) & \(16\) & \qfdata{31&255&496}{30&8&-120} & \fkdata{15&50&34}{17&20&27} & \vct{1,1,0} & 1 & \vct{2,0,3}\\*[2pt]
 & \(-8\) & \qfdata{124&175&256}{-100&64&160} & \fkdata{21&17&45}{16&38&30} & \vct{2,1,0} & 1 & \vct{1,0,3}\\*[2pt]
 & \(24\) & \qfdata{151&156&220}{-132&100&-120} & \fkdata{45&26&45}{46&32&5} & \vct{3,2,2} & 1 & \vct{2,3,0}\\*[2pt]
 & \(-8\) & \qfdata{124&151&256}{-116&-64&-32} & \fkdata{11&13&27}{50&19&49} & \vct{2,3,0} & 1 & \vct{4,2,3}\\*[2pt]
 & \(-8\) & \qfdata{156&159&199}{-84&-36&-138} & \fkdata{22&7&34}{34&31&48} & \vct{2,3,3} & 1 & \vct{3,4,1}\\
\addlinespace[0.3em]
\(G_{4;+,-,+}^{(-2)}\) & \(8\) & \qfdata{31&255&496}{30&8&-120} & 1 & \vct{1,1,0} & \vct{2,0,2} & 1\\*[2pt]
 & \(-8\) & \qfdata{124&175&256}{-100&64&160} & 1 & \vct{2,1,0} & \vct{2,2,1} & 1\\*[2pt]
 & \(-8\) & \qfdata{79&111&496}{-66&56&-72} & 1 & \vct{3,3,0} & \vct{2,0,2} & 1\\*[2pt]
 & \(8\) & \qfdata{151&156&220}{-132&100&-120} & 1 & \vct{3,2,2} & \vct{2,0,1} & 1\\*[2pt]
 & \(-16\) & \qfdata{39&199&496}{-18&24&-104} & 1 & \vct{1,1,0} & \vct{0,2,1} & 1\\*[2pt]
 & \(-8\) & \qfdata{124&151&256}{-116&-64&-32} & 1 & \vct{2,3,0} & \vct{2,1,2} & 1\\*[2pt]
 & \(8\) & \qfdata{55&151&496}{50&40&88} & 1 & \vct{3,1,0} & \vct{2,2,1} & 1\\
\addlinespace[0.3em]
\(G_{4;+,+,-}^{(-2)}\)\newline \(c_0=(0,14,1)^t\) & \(8\) & \qfdata{31&255&496}{30&8&-120} & \fkdata{8&5&5}{3&2&0} & \vct{1,1,0} & 1 & \vct{2,0,3}\\*[2pt]
 & \(-8\) & \qfdata{124&175&256}{-100&64&160} & \fkdata{2&10&13}{15&2&4} & \vct{2,1,0} & 1 & \vct{1,0,3}\\*[2pt]
 & \(-8\) & \qfdata{79&111&496}{-66&56&-72} & \fkdata{14&0&4}{10&10&16} & \vct{3,3,0} & 1 & \vct{1,3,2}\\*[2pt]
 & \(8\) & \qfdata{151&156&220}{-132&-100&120} & \fkdata{11&5&16}{0&2&10} & \vct{1,2,2} & 1 & \vct{3,2,0}\\*[2pt]
 & \(-16\) & \qfdata{39&199&496}{-18&-24&104} & \fkdata{9&6&16}{16&3&9} & \vct{3,3,0} & 1 & \vct{1,4,2}\\*[2pt]
 & \(-8\) & \qfdata{124&151&256}{-116&-64&-32} & \fkdata{6&9&12}{5&11&2} & \vct{2,3,0} & 1 & \vct{4,2,3}\\*[2pt]
 & \(8\) & \qfdata{55&151&496}{50&-40&-88} & \fkdata{0&9&1}{1&13&2} & \vct{3,1,0} & 1 & \vct{0,3,3}\\
\addlinespace[0.3em]
\(G_{4;-,-,+}^{(2)}\)\newline \(c_0=(9,9,1)^t\) & \(16\) & \qfdata{31&255&496}{30&8&-120} & \fkdata{2&4&5}{8&7&0} & \vct{1,1,0} & 1 & 1\\*[2pt]
 & \(8\) & \qfdata{124&175&256}{-100&64&160} & \fkdata{3&1&7}{0&9&5} & \vct{2,1,0} & 1 & 1\\*[2pt]
 & \(-8\) & \qfdata{79&111&496}{-66&56&-72} & \fkdata{8&10&5}{2&5&0} & \vct{3,3,0} & 1 & 1\\*[2pt]
 & \(-8\) & \qfdata{151&156&220}{-132&-100&120} & \fkdata{9&6&1}{6&4&10} & \vct{3,2,2} & 1 & 1\\*[2pt]
 & \(8\) & \qfdata{39&199&496}{-18&-24&104} & \fkdata{0&9&3}{10&10&3} & \vct{1,1,0} & 1 & 1\\*[2pt]
 & \(8\) & \qfdata{55&151&496}{50&40&88} & \fkdata{8&10&7}{6&8&10} & \vct{3,1,0} & 1 & 1\\*[2pt]
 & \(8\) & \qfdata{156&159&199}{-84&-36&-138} & \fkdata{0&8&9}{1&6&2} & \vct{2,3,3} & 1 & 1\\
\addlinespace[0.3em]
\(G_{4;-,+,-}^{(2)}\)\newline \(c_0=(37,2,1)^t\) & \(16\) & \qfdata{31&255&496}{30&8&-120} & \fkdata{7&27&23}{31&3&2} & \vct{1,1,0} & \vct{2,0,2} & \vct{2,0,3}\\*[2pt]
 & \(8\) & \qfdata{124&175&256}{-100&64&160} & \fkdata{8&13&30}{22&17&14} & \vct{2,1,0} & \vct{2,2,1} & \vct{1,0,3}\\*[2pt]
 & \(-8\) & \qfdata{79&111&496}{-66&56&-72} & \fkdata{19&9&4}{23&6&17} & \vct{3,3,0} & \vct{2,0,2} & \vct{1,3,2}\\*[2pt]
 & \(-8\) & \qfdata{151&156&220}{-132&100&-120} & \fkdata{12&31&42}{16&17&31} & \vct{1,2,2} & \vct{1,0,2} & \vct{3,2,0}\\*[2pt]
 & \(8\) & \qfdata{39&199&496}{18&-24&-104} & \fkdata{31&16&10}{15&5&10} & \vct{3,1,0} & \vct{0,2,1} & \vct{1,1,2}\\*[2pt]
 & \(8\) & \qfdata{55&151&496}{-50&40&-88} & \fkdata{25&0&24}{15&20&22} & \vct{1,1,0} & \vct{1,2,2} & \vct{0,3,3}\\*[2pt]
 & \(8\) & \qfdata{156&159&199}{-84&-36&-138} & \fkdata{41&33&28}{17&23&34} & \vct{2,3,3} & \vct{0,0,1} & \vct{3,4,1}\\
\addlinespace[0.3em]
\(G_{4;+,-,+}^{(2)}\)\newline \(c_0=(168,11,1)^t\) & \(32\) & \qfdata{31&255&496}{30&8&-120} & \fkdata{132&133&60}{49&68&98} & \vct{1,1,0} & 1 & 1\\*[2pt]
 & \(-32\) & \qfdata{151&156&220}{-132&-100&120} & \fkdata{21&125&141}{183&152&178} & \vct{3,2,2} & 1 & 1\\
\addlinespace[0.3em]
\(G_{4;+,+,-}^{(2)}\)\newline \(c_0=(99,2,1)^t\) & \(32\) & \qfdata{31&255&496}{30&8&-120} & \fkdata{85&65&1}{1&24&3} & \vct{1,1,0} & \vct{2,0,2} & \vct{2,0,3}\\*[2pt]
 & \(-32\) & \qfdata{151&156&220}{-132&100&-120} & \fkdata{86&29&28}{16&58&60} & \vct{3,2,2} & \vct{2,0,1} & \vct{2,3,0}\\
\addlinespace[0.3em]
\(G_{4;-,-,+}^{(1)}\)\newline \(c_0=(53,1,1)^t\) & \(2\) & \qfdata{60&111&175}{60&60&30} & \fkdata{12&47&43}{42&52&35} & 1 & \vct{0,0,1} & \vct{0,3,0}\\*[2pt]
 & \(-4\) & \qfdata{55&151&151}{50&50&-82} & \fkdata{33&15&7}{57&16&2} & 1 & \vct{1,1,1} & \vct{0,3,2}\\*[2pt]
 & \(4\) & \qfdata{60&79&271}{-60&60&34} & \fkdata{58&19&5}{1&5&6} & 1 & \vct{0,2,1} & \vct{0,4,2}\\*[2pt]
 & \(-4\) & \qfdata{39&156&199}{-36&18&-156} & \fkdata{4&32&9}{9&27&22} & 1 & \vct{0,0,1} & \vct{1,3,1}\\*[2pt]
 & \(-2\) & \qfdata{15&256&256}{0&0&128} & \fkdata{4&37&28}{53&42&11} & 1 & \vct{0,2,2} & \vct{0,3,2}\\
\addlinespace[0.3em]
\(G_{4;-,+,-}^{(1)}\)\newline \(c_0=(29,10,1)^t\) & \(2\) & \qfdata{60&111&175}{-60&60&-30} & \fkdata{35&23&15}{20&47&12} & 1 & 1 & 1\\*[2pt]
 & \(-4\) & \qfdata{55&151&151}{50&-50&82} & \fkdata{5&47&30}{33&35&28} & 1 & 1 & 1\\*[2pt]
 & \(4\) & \qfdata{60&79&271}{-60&60&34} & \fkdata{27&49&16}{16&52&39} & 1 & 1 & 1\\*[2pt]
 & \(-4\) & \qfdata{39&156&199}{-36&18&-156} & \fkdata{1&41&30}{8&13&12} & 1 & 1 & 1\\*[2pt]
 & \(-2\) & \qfdata{15&256&256}{0&0&128} & \fkdata{8&50&47}{25&26&38} & 1 & 1 & 1\\
\addlinespace[0.3em]
\(G_{4;+,-,-}^{(1)}\)\newline \(c_0=(38,13,1)^t\) & \(8\) & \qfdata{60&111&175}{-60&60&-30} & \fkdata{24&66&10}{62&62&26} & \vct{2,1,3} & \vct{0,0,1} & 1\\*[2pt]
 & \(-8\) & \qfdata{55&151&151}{-50&50&82} & \fkdata{68&52&70}{44&58&56} & \vct{1,1,3} & \vct{2,1,2} & 1\\
\addlinespace[0.3em]
\(G_{4;+,+,+}^{(1)}\)\newline \(c_0=(39,2,1)^t\) & \(16\) & \qfdata{60&111&175}{60&-60&-30} & \fkdata{4&15&19}{38&40&23} & \vct{2,1,3} & 1 & \vct{0,2,0}\\
\addlinespace[0.3em]
\(G_{4;-,-,+}^{(-1)}\)\newline \(c_0=(10,10,1)^t\) & \(8\) & \qfdata{31&124&255}{-4&-30&-60} & \fkdata{4&1&1}{9&4&4} & 1 & 1 & \vct{3,4,0}\\*[2pt]
 & \(8\) & \qfdata{79&111&124}{-66&-28&36} & \fkdata{7&9&4}{1&8&5} & 1 & 1 & \vct{4,2,1}\\
\addlinespace[0.3em]
\(G_{4;-,+,-}^{(-1)}\)\newline \(c_0=(40,1,1)^t\) & \(8\) & \qfdata{31&124&255}{-4&-30&-60} & \fkdata{34&38&5}{16&23&26} & 1 & \vct{2,2,0} & 1\\*[2pt]
 & \(8\) & \qfdata{79&111&124}{-66&-28&36} & \fkdata{20&21&21}{3&5&29} & 1 & \vct{1,0,1} & 1\\
\addlinespace[0.3em]
\(G_{4;+,-,-}^{(-1)}\)\newline \(c_0=(16,1,1)^t\) & \(8\) & \qfdata{31&124&255}{-4&-30&-60} & \fkdata{5&7&20}{0&22&13} & \vct{1,2,3} & 1 & 1\\*[2pt]
 & \(8\) & \qfdata{79&111&124}{-66&-28&36} & \fkdata{3&10&21}{1&2&0} & \vct{3,3,2} & 1 & 1\\
\addlinespace[0.3em]
\(G_{4;+,+,+}^{(-1)}\)\newline \(c_0=(17,6,1)^t\) & \(8\) & \qfdata{31&124&255}{4&-30&60} & \fkdata{9&10&19}{28&5&20} & \vct{1,2,3} & \vct{2,1,0} & \vct{3,1,0}\\*[2pt]
 & \(8\) & \qfdata{79&111&124}{-66&-28&36} & \fkdata{21&24&16}{14&1&15} & \vct{3,3,2} & \vct{1,0,1} & \vct{4,2,1}\\
\addlinespace[0.3em]
\(G_{4;-,-,-}^{(-2)}\)\newline \(c_0=(68,8,1)^t\) & \(16\) & \qfdata{31&255&496}{30&8&-120} & \fkdata{31&39&38}{21&48&45} & \vct{1,1,0} & \vct{2,0,2} & \vct{2,0,3}\\*[2pt]
 & \(-8\) & \qfdata{124&175&256}{-100&64&160} & \fkdata{77&11&67}{55&63&59} & \vct{2,1,0} & \vct{2,2,1} & \vct{1,0,3}\\*[2pt]
 & \(24\) & \qfdata{151&156&220}{-132&100&-120} & \fkdata{64&42&4}{35&41&47} & \vct{3,2,2} & \vct{2,0,1} & \vct{2,3,0}\\*[2pt]
 & \(-8\) & \qfdata{124&151&256}{-116&-64&-32} & \fkdata{54&74&43}{63&17&75} & \vct{2,3,0} & \vct{2,1,2} & \vct{4,2,3}\\*[2pt]
 & \(-8\) & \qfdata{156&159&199}{-84&-36&-138} & \fkdata{21&22&44}{47&29&14} & \vct{2,3,3} & \vct{0,0,1} & \vct{3,4,1}\\
\addlinespace[0.3em]
\(G_{4;-,+,+}^{(-2)}\)\newline \(c_0=(12,1,1)^t\) & \(16\) & \qfdata{31&255&496}{30&8&-120} & \fkdata{12&4&7}{2&5&10} & \vct{1,1,0} & 1 & 1\\*[2pt]
 & \(-8\) & \qfdata{124&175&256}{-100&64&160} & \fkdata{14&2&11}{16&1&14} & \vct{2,1,0} & 1 & 1\\*[2pt]
 & \(24\) & \qfdata{151&156&220}{-132&100&-120} & \fkdata{14&6&18}{17&3&9} & \vct{3,2,2} & 1 & 1\\*[2pt]
 & \(-8\) & \qfdata{124&151&256}{-116&-64&-32} & \fkdata{4&6&17}{4&5&4} & \vct{2,3,0} & 1 & 1\\*[2pt]
 & \(-8\) & \qfdata{156&159&199}{-84&-36&-138} & \fkdata{2&9&5}{6&16&3} & \vct{2,3,3} & 1 & 1\\
\addlinespace[0.3em]
\(G_{4;+,-,-}^{(-2)}\)\newline \(c_0=(11,15,1)^t\) & \(16\) & \qfdata{31&255&496}{30&8&-120} & \fkdata{2&11&12}{20&0&12} & \vct{1,1,0} & \vct{2,0,2} & \vct{2,0,3}\\*[2pt]
 & \(-8\) & \qfdata{79&111&496}{-66&56&-72} & \fkdata{15&13&19}{1&15&18} & \vct{3,3,0} & \vct{2,0,2} & \vct{1,3,2}\\*[2pt]
 & \(16\) & \qfdata{151&156&220}{-132&-100&120} & \fkdata{2&16&15}{3&7&9} & \vct{1,2,2} & \vct{1,0,1} & \vct{3,2,0}\\*[2pt]
 & \(-16\) & \qfdata{39&199&496}{-18&-24&104} & \fkdata{11&4&12}{5&19&14} & \vct{3,3,0} & \vct{0,1,1} & \vct{1,4,2}\\*[2pt]
 & \(-8\) & \qfdata{124&151&256}{-116&-64&-32} & \fkdata{5&8&14}{1&11&18} & \vct{2,3,0} & \vct{2,1,2} & \vct{4,2,3}\\
\addlinespace[0.3em]
\(G_{4;+,+,+}^{(-2)}\)\newline \(c_0=(23,1,1)^t\) & \(8\) & \qfdata{31&255&496}{30&8&-120} & \fkdata{0&30&8}{1&0&0} & \vct{1,1,0} & 1 & 1\\*[2pt]
 & \(-8\) & \qfdata{124&175&256}{-100&64&160} & \fkdata{13&18&14}{25&16&16} & \vct{2,1,0} & 1 & 1\\*[2pt]
 & \(-8\) & \qfdata{79&111&496}{-66&56&-72} & \fkdata{14&29&25}{13&30&10} & \vct{3,3,0} & 1 & 1\\*[2pt]
 & \(8\) & \qfdata{151&156&220}{-132&100&-120} & \fkdata{7&8&0}{22&4&13} & \vct{3,2,2} & 1 & 1\\*[2pt]
 & \(-16\) & \qfdata{39&199&496}{-18&24&-104} & \fkdata{24&8&24}{12&24&11} & \vct{1,1,0} & 1 & 1\\*[2pt]
 & \(-8\) & \qfdata{124&151&256}{-116&-64&-32} & \fkdata{29&23&12}{17&30&28} & \vct{2,3,0} & 1 & 1\\*[2pt]
 & \(8\) & \qfdata{55&151&496}{50&40&88} & \fkdata{16&26&29}{26&10&28} & \vct{3,1,0} & 1 & 1\\
\addlinespace[0.3em]
\(G_{4;-,-,-}^{(2)}\)\newline \(c_0=(7,0,1)^t\) & \(16\) & \qfdata{31&255&496}{30&8&-120} & \fkdata{4&1&4}{8&7&0} & \vct{1,1,0} & 1 & \vct{2,0,3}\\*[2pt]
 & \(8\) & \qfdata{124&175&256}{-100&64&160} & \fkdata{6&6&8}{10&1&8} & \vct{2,1,0} & 1 & \vct{1,0,3}\\*[2pt]
 & \(-8\) & \qfdata{79&111&496}{-66&56&-72} & \fkdata{0&3&2}{7&2&10} & \vct{3,3,0} & 1 & \vct{1,3,2}\\*[2pt]
 & \(-8\) & \qfdata{151&156&220}{-132&-100&120} & \fkdata{3&1&7}{11&0&12} & \vct{1,2,2} & 1 & \vct{3,2,0}\\*[2pt]
 & \(8\) & \qfdata{39&199&496}{-18&-24&104} & \fkdata{8&3&5}{12&3&5} & \vct{3,3,0} & 1 & \vct{1,4,2}\\*[2pt]
 & \(8\) & \qfdata{55&151&496}{50&-40&-88} & \fkdata{10&5&6}{1&11&0} & \vct{3,1,0} & 1 & \vct{0,3,3}\\*[2pt]
 & \(8\) & \qfdata{156&159&199}{-84&-36&-138} & \fkdata{1&4&6}{0&8&12} & \vct{2,3,3} & 1 & \vct{3,4,1}\\
\addlinespace[0.3em]
\(G_{4;-,+,+}^{(2)}\)\newline \(c_0=(17,2,1)^t\) & \(16\) & \qfdata{31&255&496}{30&8&-120} & \fkdata{27&21&13}{3&23&28} & \vct{1,1,0} & \vct{2,0,2} & 1\\*[2pt]
 & \(8\) & \qfdata{124&175&256}{-100&64&160} & \fkdata{2&3&18}{22&6&3} & \vct{2,1,0} & \vct{2,2,1} & 1\\*[2pt]
 & \(-8\) & \qfdata{79&111&496}{-66&56&-72} & \fkdata{28&12&6}{10&15&20} & \vct{3,3,0} & \vct{2,0,2} & 1\\*[2pt]
 & \(-8\) & \qfdata{151&156&220}{-132&100&-120} & \fkdata{6&1&17}{25&11&11} & \vct{3,2,2} & \vct{2,0,1} & 1\\*[2pt]
 & \(8\) & \qfdata{39&199&496}{-18&24&-104} & \fkdata{12&8&28}{20&11&9} & \vct{1,1,0} & \vct{0,2,1} & 1\\*[2pt]
 & \(8\) & \qfdata{55&151&496}{50&40&88} & \fkdata{7&27&9}{20&20&5} & \vct{3,1,0} & \vct{2,2,1} & 1\\*[2pt]
 & \(8\) & \qfdata{156&159&199}{-84&-36&-138} & \fkdata{2&27&20}{21&13&5} & \vct{2,3,3} & \vct{0,0,1} & 1\\
\addlinespace[0.3em]
\(G_{4;+,-,-}^{(2)}\)\newline \(c_0=(68,1,1)^t\) & \(32\) & \qfdata{31&255&496}{30&8&-120} & \fkdata{59&11&65}{12&23&38} & \vct{1,1,0} & 1 & \vct{2,0,3}\\*[2pt]
 & \(-32\) & \qfdata{151&156&220}{-132&-100&120} & \fkdata{16&16&46}{30&48&30} & \vct{3,2,2} & 1 & \vct{2,3,0}\\
\addlinespace[0.3em]
\(G_{4;+,+,+}^{(2)}\)\newline \(c_0=(402,22,1)^t\) & \(32\) & \qfdata{31&255&496}{30&8&-120} & \fkdata{265&0&47}{414&425&226} & \vct{1,1,0} & \vct{2,0,2} & 1\\*[2pt]
 & \(-32\) & \qfdata{151&156&220}{-132&100&-120} & \fkdata{355&385&338}{359&52&115} & \vct{3,2,2} & \vct{2,0,1} & 1\\
\addlinespace[0.3em]
\end{longtable}
\endgroup

\section{Fourier expansions and Sturm comparison data}
\label{app:fourier-data}

Tables~\ref{tab:qexp-three} and~\ref{tab:qexp-four} record initial
Fourier expansions in the normalizations of
Appendix~\ref{app:weighted-formulas}.
Tables~\ref{tab:sturm-comparison-three} and~\ref{tab:sturm-comparison-four}
give the comparisons used in Section~\ref{sec:prime-noncongruence}.

\begingroup
\scriptsize
\setlength{\LTleft}{0pt}
\setlength{\LTright}{\fill}
\setlength{\LTcapwidth}{\textwidth}
\setlength{\tabcolsep}{3pt}
\renewcommand{\arraystretch}{1.02}
\begin{longtable}{@{}>{\raggedright\arraybackslash}p{0.23\textwidth}>{\raggedright\arraybackslash}p{0.735\textwidth}@{}}
\caption{Initial Fourier expansions for \(\cos\theta=\pm 3/5\), through \(q^{150}\).}\label{tab:qexp-three}\\
\toprule
Form & Expansion\\
\midrule
\endfirsthead
\multicolumn{2}{@{}l}{\textit{Table~\thetable\ (continued)}}\\
\toprule
Form & Expansion\\
\midrule
\endhead
\midrule
\multicolumn{2}{r@{}}{\textit{Continued on the next page}}\\
\endfoot
\bottomrule
\endlastfoot
\(G_{3;-,+}^{(1)}\) & \(-4q^{5}\allowbreak -8q^{29}\allowbreak +4q^{45}\allowbreak +8q^{61}\allowbreak +16q^{69}\allowbreak +8q^{85}\allowbreak -8q^{101}\allowbreak +8q^{109}\allowbreak +4q^{125}\allowbreak -16q^{141}\allowbreak -24q^{149}\allowbreak +O(q^{151})\)\\
\(G_{3;+,+}^{(1)}\) & \(4q\allowbreak +4q^{9}\allowbreak -4q^{25}\allowbreak -12q^{49}\allowbreak +8q^{65}\allowbreak -12q^{81}\allowbreak -12q^{121}\allowbreak +16q^{129}\allowbreak -8q^{145}\allowbreak +O(q^{151})\)\\
\(G_{3;-,+}^{(-1)}\) & \(-16q^{11}\allowbreak -16q^{19}\allowbreak +16q^{35}\allowbreak -32q^{51}\allowbreak +16q^{59}\allowbreak -32q^{91}\allowbreak +16q^{99}\allowbreak +16q^{115}\allowbreak -16q^{131}\allowbreak +16q^{139}\allowbreak +O(q^{151})\)\\
\(G_{3;+,+}^{(-1)}\) & \(-16q^{15}\allowbreak -32q^{39}\allowbreak -16q^{55}\allowbreak +32q^{71}\allowbreak -16q^{95}\allowbreak +32q^{111}\allowbreak +O(q^{151})\)\\
\(G_{3;-,+}^{(-2)}\) & \(16iq^{11}\allowbreak -16q^{19}\allowbreak +32iq^{51}\allowbreak -16q^{59}\allowbreak +16q^{99}\allowbreak -16iq^{131}\allowbreak -16q^{139}\allowbreak +O(q^{151})\)\\
\(G_{3;+,-}^{(-2)}\) & \(-4q^{7}\allowbreak +4q^{15}\allowbreak +4q^{23}\allowbreak -4q^{47}\allowbreak -4q^{63}\allowbreak +8q^{87}\allowbreak -8q^{95}\allowbreak -4q^{103}\allowbreak +12q^{127}\allowbreak +8q^{143}\allowbreak +O(q^{151})\)\\
\(G_{3;-,+}^{(2)}\) & \(16iq^{21}\allowbreak -16q^{29}\allowbreak +16q^{69}\allowbreak +16iq^{101}\allowbreak -32q^{109}\allowbreak -16iq^{141}\allowbreak +32q^{149}\allowbreak +O(q^{151})\)\\
\(G_{3;+,+}^{(2)}\) & \(4q\allowbreak +4iq^{9}\allowbreak +16q^{41}\allowbreak +12iq^{49}\allowbreak +12q^{81}\allowbreak -20q^{121}\allowbreak -16iq^{129}\allowbreak +O(q^{151})\)\\
\(G_{3;-,-}^{(1)}\) & \(-8q^{3}\allowbreak +8q^{43}\allowbreak -8iq^{67}\allowbreak +24q^{83}\allowbreak +8iq^{107}\allowbreak -16q^{123}\allowbreak +24iq^{147}\allowbreak +O(q^{151})\)\\
\(G_{3;+,-}^{(1)}\) & \(4q^{7}\allowbreak -4iq^{23}\allowbreak -4q^{47}\allowbreak +4iq^{63}\allowbreak -16q^{87}\allowbreak +12iq^{103}\allowbreak +12q^{127}\allowbreak -8iq^{143}\allowbreak +O(q^{151})\)\\
\(G_{3;-,-}^{(-1)}\) & \(16iq^{13}\allowbreak +16iq^{53}\allowbreak -16q^{77}\allowbreak -16q^{117}\allowbreak +16iq^{133}\allowbreak +O(q^{151})\)\\
\(G_{3;+,-}^{(-1)}\) & \(-32q^{17}\allowbreak +32iq^{33}\allowbreak -32q^{57}\allowbreak +32iq^{73}\allowbreak +32q^{97}\allowbreak +O(q^{151})\)\\
\(G_{3;-,-}^{(-2)}\) & \(-8q^{5}\allowbreak +16q^{13}\allowbreak -16q^{37}\allowbreak +8q^{45}\allowbreak -16q^{53}\allowbreak +32q^{77}\allowbreak +16q^{85}\allowbreak -32q^{93}\allowbreak +16q^{117}\allowbreak -8q^{125}\allowbreak +O(q^{151})\)\\
\(G_{3;+,+}^{(-2)}\) & \(4q\allowbreak +4iq^{9}\allowbreak -16q^{41}\allowbreak -20iq^{49}\allowbreak +12q^{81}\allowbreak +32iq^{89}\allowbreak +12q^{121}\allowbreak -16iq^{129}\allowbreak +O(q^{151})\)\\
\(G_{3;-,-}^{(2)}\) & \(16q^{3}\allowbreak -16q^{35}\allowbreak -16q^{43}\allowbreak -16q^{67}\allowbreak +16q^{75}\allowbreak -48q^{83}\allowbreak +16q^{107}\allowbreak +48q^{115}\allowbreak +32q^{123}\allowbreak +48q^{147}\allowbreak +O(q^{151})\)\\
\(G_{3;+,-}^{(2)}\) & \(4q^{7}\allowbreak +4q^{15}\allowbreak -4q^{23}\allowbreak -12q^{47}\allowbreak -8q^{55}\allowbreak +4q^{63}\allowbreak +8q^{87}\allowbreak +4q^{103}\allowbreak +4q^{127}\allowbreak +8q^{143}\allowbreak +O(q^{151})\)\\
\end{longtable}
\endgroup

\begingroup
\scriptsize
\setlength{\LTleft}{0pt}
\setlength{\LTright}{\fill}
\setlength{\LTcapwidth}{\textwidth}
\setlength{\tabcolsep}{3pt}
\renewcommand{\arraystretch}{1.02}
\begin{longtable}{@{}>{\raggedright\arraybackslash}p{0.23\textwidth}>{\raggedright\arraybackslash}p{0.735\textwidth}@{}}
\caption{Initial Fourier expansions for \(\cos\theta=\pm 4/5\), through \(q^{150}\).}\label{tab:qexp-four}\\
\toprule
Form & Expansion\\
\midrule
\endfirsthead
\multicolumn{2}{@{}l}{\textit{Table~\thetable\ (continued)}}\\
\toprule
Form & Expansion\\
\midrule
\endhead
\midrule
\multicolumn{2}{r@{}}{\textit{Continued on the next page}}\\
\endfoot
\bottomrule
\endlastfoot
\(G_{4;-,-,-}^{(1)}\) & \(-8q^{5}\allowbreak +16q^{53}\allowbreak +8q^{125}\allowbreak +O(q^{151})\)\\
\(G_{4;-,+,+}^{(1)}\) & \(-16q^{21}\allowbreak -32q^{61}\allowbreak -16iq^{69}\allowbreak -32iq^{109}\allowbreak -16q^{141}\allowbreak +O(q^{151})\)\\
\(G_{4;+,-,+}^{(1)}\) & \(128iq^{41}\allowbreak -128q^{89}\allowbreak +O(q^{151})\)\\
\(G_{4;+,+,-}^{(1)}\) & \(-16q^{33}\allowbreak -16q^{57}\allowbreak +32q^{73}\allowbreak +32q^{97}\allowbreak -16q^{105}\allowbreak +32q^{145}\allowbreak +O(q^{151})\)\\
\(G_{4;-,-,-}^{(-1)}\) & \(16q^{3}\allowbreak +16q^{27}\allowbreak -32q^{43}\allowbreak -32q^{67}\allowbreak +16q^{75}\allowbreak -32q^{115}\allowbreak +32q^{123}\allowbreak +48q^{147}\allowbreak +O(q^{151})\)\\
\(G_{4;-,+,+}^{(-1)}\) & \(32q^{11}\allowbreak +32iq^{59}\allowbreak +32q^{131}\allowbreak +O(q^{151})\)\\
\(G_{4;+,-,+}^{(-1)}\) & \(-16iq^{31}\allowbreak +16q^{79}\allowbreak -16iq^{111}\allowbreak +O(q^{151})\)\\
\(G_{4;+,+,-}^{(-1)}\) & \(32q^{23}\allowbreak -32q^{47}\allowbreak -32q^{95}\allowbreak -64q^{143}\allowbreak +O(q^{151})\)\\
\(G_{4;-,-,+}^{(-2)}\) & \(32q^{19}\allowbreak -64q^{91}\allowbreak +32q^{115}\allowbreak -32q^{139}\allowbreak +O(q^{151})\)\\
\(G_{4;-,+,-}^{(-2)}\) & \(-32q^{3}\allowbreak +32iq^{27}\allowbreak -64q^{83}\allowbreak +64iq^{107}\allowbreak -64q^{123}\allowbreak +96iq^{147}\allowbreak +O(q^{151})\)\\
\(G_{4;+,-,+}^{(-2)}\) & \(-16q^{31}\allowbreak +16q^{55}\allowbreak -16q^{79}\allowbreak +O(q^{151})\)\\
\(G_{4;+,+,-}^{(-2)}\) & \(32iq^{23}\allowbreak +32q^{47}\allowbreak +64iq^{143}\allowbreak +O(q^{151})\)\\
\(G_{4;-,-,+}^{(2)}\) & \(16q^{5}\allowbreak +32q^{29}\allowbreak -16q^{45}\allowbreak -32q^{69}\allowbreak +32q^{101}\allowbreak -16q^{125}\allowbreak -32q^{141}\allowbreak -32q^{149}\allowbreak +O(q^{151})\)\\
\(G_{4;-,+,-}^{(2)}\) & \(32q^{13}\allowbreak -32iq^{37}\allowbreak +O(q^{151})\)\\
\(G_{4;+,-,+}^{(2)}\) & \(-64q^{65}\allowbreak -128q^{89}\allowbreak +64q^{105}\allowbreak +128q^{129}\allowbreak +O(q^{151})\)\\
\(G_{4;+,+,-}^{(2)}\) & \(-64q^{73}\allowbreak +64iq^{97}\allowbreak +O(q^{151})\)\\
\(G_{4;-,-,+}^{(1)}\) & \(-32q^{19}\allowbreak -64iq^{91}\allowbreak +32q^{139}\allowbreak +O(q^{151})\)\\
\(G_{4;-,+,-}^{(1)}\) & \(-8q^{3}\allowbreak +8q^{27}\allowbreak +16q^{35}\allowbreak -8q^{75}\allowbreak -16q^{83}\allowbreak -16q^{107}\allowbreak -16q^{123}\allowbreak +40q^{147}\allowbreak +O(q^{151})\)\\
\(G_{4;+,-,-}^{(1)}\) & \(-32q^{7}\allowbreak +32q^{55}\allowbreak +32q^{103}\allowbreak -32q^{127}\allowbreak +O(q^{151})\)\\
\(G_{4;+,+,+}^{(1)}\) & \(64q^{39}\allowbreak -128iq^{71}\allowbreak +O(q^{151})\)\\
\(G_{4;-,-,+}^{(-1)}\) & \(-16iq^{21}\allowbreak +32q^{29}\allowbreak -16q^{69}\allowbreak -32iq^{101}\allowbreak -48iq^{141}\allowbreak +64q^{149}\allowbreak +O(q^{151})\)\\
\(G_{4;-,+,-}^{(-1)}\) & \(-32q^{37}\allowbreak +32q^{85}\allowbreak +32q^{133}\allowbreak +O(q^{151})\)\\
\(G_{4;+,-,-}^{(-1)}\) & \(-16q^{33}\allowbreak +16q^{57}\allowbreak +32q^{65}\allowbreak -16q^{105}\allowbreak -32q^{113}\allowbreak -32q^{137}\allowbreak +O(q^{151})\)\\
\(G_{4;+,+,+}^{(-1)}\) & \(16iq\allowbreak +48q^{49}\allowbreak -16iq^{121}\allowbreak +O(q^{151})\)\\
\(G_{4;-,-,-}^{(-2)}\) & \(64q^{53}\allowbreak +O(q^{151})\)\\
\(G_{4;-,+,+}^{(-2)}\) & \(-32q^{61}\allowbreak -32q^{69}\allowbreak +32q^{85}\allowbreak +96q^{109}\allowbreak -32q^{141}\allowbreak +O(q^{151})\)\\
\(G_{4;+,-,-}^{(-2)}\) & \(32iq^{17}\allowbreak +96q^{113}\allowbreak -32iq^{137}\allowbreak +O(q^{151})\)\\
\(G_{4;+,+,+}^{(-2)}\) & \(16q\allowbreak +16q^{9}\allowbreak -16q^{25}\allowbreak -48q^{49}\allowbreak +16q^{81}\allowbreak -32q^{105}\allowbreak +16q^{121}\allowbreak +32q^{145}\allowbreak +O(q^{151})\)\\
\(G_{4;-,-,-}^{(2)}\) & \(16q^{3}\allowbreak +16iq^{27}\allowbreak -32q^{43}\allowbreak +32iq^{67}\allowbreak +32q^{123}\allowbreak +80iq^{147}\allowbreak +O(q^{151})\)\\
\(G_{4;-,+,+}^{(2)}\) & \(32q^{11}\allowbreak +32q^{35}\allowbreak -32q^{59}\allowbreak -32q^{131}\allowbreak +O(q^{151})\)\\
\(G_{4;+,-,-}^{(2)}\) & \(-64iq^{7}\allowbreak +64q^{63}\allowbreak -64q^{103}\allowbreak -192iq^{127}\allowbreak +O(q^{151})\)\\
\(G_{4;+,+,+}^{(2)}\) & \(-128q^{71}\allowbreak -128q^{95}\allowbreak +128q^{119}\allowbreak +O(q^{151})\)\\
\end{longtable}
\endgroup

\begin{table}[htbp]
\centering
\footnotesize
\begin{tabular}{ccccc}
\toprule
\(\cos\theta\) & residue set \(C_r\) mod \(40\) & Form \(G_r\)
& \(e_r\) & \((Q_r,d_r)\)\\
\midrule
\(3/5\) & \(\{11,19\}\) & \(G_{3;-,+}^{(-1)}\)
& \(4\) & \((x^2+y^2+z^2,3)\)\\
\(3/5\) & \(\{21,29\}\) & \(G_{3;-,+}^{(1)}\)
& \(3\) & \((x^2+y^2+z^2,3)\)\\
\(-3/5\) & \(\{3,27\}\) & \(G_{3;-,-}^{(1)}\)
& \(3\) & \((x^2+y^2+z^2,3)\)\\
\(-3/5\) & \(\{7,23\}\) & \(G_{3;+,-}^{(1)}\)
& \(2\) & \((x^2+3y^2+5z^2,2)\)\\
\bottomrule
\end{tabular}
\caption{Comparisons through the Sturm bound for the \(\pm3/5\) family.}
\label{tab:sturm-comparison-three}
\end{table}

\begin{table}[htbp]
\centering
\footnotesize
\begin{tabular}{ccccc}
\toprule
\(\cos\theta\) & residue set \(C_r\) mod \(120\) & Form \(G_r\)
& \(e_r\) & \((Q_r,d_r)\)\\
\midrule
\(4/5\) & \(\{53,77\}\) & \(G_{4;-,-,-}^{(1)}\)
& \(4\) & \((x^2+y^2+z^2,3)\)\\
\(4/5\) & \(\{43,67\}\) & \(G_{4;-,-,-}^{(-1)}\)
& \(5\) & \((x^2+y^2+z^2,3)\)\\
\(4/5\) & \(\{11,59\}\) & \(G_{4;-,+,+}^{(-1)}\)
& \(5\) & \((x^2+y^2+z^2,3)\)\\
\(4/5\) & \(\{31,79\}\) & \(G_{4;+,-,+}^{(-1)}\)
& \(4\) & \((x^2+3y^2+5z^2,2)\)\\
\(4/5\) & \(\{23,47\}\) & \(G_{4;+,+,-}^{(-1)}\)
& \(5\) & \((x^2+3y^2+5z^2,2)\)\\
\(-4/5\) & \(\{19,91\}\) & \(G_{4;-,-,+}^{(1)}\)
& \(5\) & \((x^2+y^2+z^2,3)\)\\
\(-4/5\) & \(\{83,107\}\) & \(G_{4;-,+,-}^{(1)}\)
& \(4\) & \((x^2+y^2+z^2,3)\)\\
\(-4/5\) & \(\{7,103\}\) & \(G_{4;+,-,-}^{(1)}\)
& \(5\) & \((x^2+3y^2+5z^2,2)\)\\
\bottomrule
\end{tabular}
\caption{Comparisons through the Sturm bound for the \(\pm4/5\) family.}
\label{tab:sturm-comparison-four}
\end{table}

\end{document}